\documentclass[10pt,a4paper]{article}
\usepackage[T1]{fontenc}
\usepackage[utf8]{inputenc}
\usepackage{authblk}
\usepackage{amsmath,amssymb,amsthm,esint,bm}
\usepackage{mathrsfs}
\usepackage{bookmark}
\usepackage{amsmath}
\allowdisplaybreaks[3]

\newtheorem{definition}{Definition}[section]
\newtheorem{theorem}[definition]{Theorem}
\newtheorem{lemma}[definition]{Lemma}

\newtheorem{corollary}[definition]{Corollary}
\theoremstyle{remark}
\newtheorem{remark}[definition]{Remark}
\numberwithin{equation}{section}

\newcommand{\R}{\mathbb{R}}

\allowdisplaybreaks[4]

\title{H\"{o}lder Continuity for Fully Fractional Parabolic Equations with Space-time Nonlocal Operators}

\author[a]{Lingwei Ma}
\author[b]{Qi Xiong}
\author[c]{Zhenqiu Zhang\thanks{Corresponding author.}}

\affil[a]{School of Mathematical Sciences, Nankai University, Tianjin 300071, P.R. China}
\affil[b]{School of Mathematics, Southwest Jiaotong University, Chengdu 610031, Sichuan, P.R. China}
\affil[c]{School of Mathematical Sciences and LPMC, Nankai University, Tianjin 300071, P.R. China}

\date{\today}

\usepackage{hyperref}
\begin{document}
\maketitle
\footnotetext[1]{E-mail: mlw1103@163.com (L. Ma), xq@swjtu.edu.cn (Q. Xiong), zqzhang@nankai.edu.cn (Z. Zhang).}
\begin{abstract}
We study the local H\"{o}lder regularity of weak solutions to the fully fractional parabolic equations involving spatial fractional diffusion and fractional time derivatives of the Marchaud type. It is worth noting that we do not impose boundedness assumptions on the weak solutions and nonhomogeneous terms. Within the space-time nonlocal framework, it is crucial to consider both space-dependent nonlocal tail terms and the first introduced time-dependent nonlocal tail term. By adapting a nonlocal variant of the parabolic De Giorgi iterative technique, we initially establish a priori local boundedness with tail terms for weak solutions and then prove the local H\"{o}lder continuity.
\\

Mathematics Subject classification (2010):  35R11; 26A33; 35B65; 39A22.

Keywords:  Space-time nonlocal parabolic equation; Local boundedness; H\"{o}lder regularity; Nonlocal tails; De Giorgi technique. \\

\end{abstract}


\section{Introduction.}\label{section1}
In this paper, we investigate the space-time nonlocal parabolic equation involving the fractional time derivative and the fractional diffusion in space as follows
\begin{equation}\label{eq1}
\partial^s_t u(t,x) +\mathcal{L} u(t,x) = f(t,x) ~~\text{in}~~I \times \Omega,
\end{equation}
where $I\subset \R$ is a finite open interval, $\Omega\subset\R^d$
is a bounded domain, and the dimension $d \geq2$.
The fractional time derivative $\partial_t^s$ we consider here is the Marchaud fractional derivative of order $s\in (0,1)$, defined by
\begin{equation}\label{fratime}
\partial_t^s u(t,x):=C_{s}\int_{-\infty}^t \frac{u(t, x)-u(\tau ,x)}{(t-\tau)^{1+s}}\operatorname{d}\!\tau,
\end{equation}
which was first introduced by Marchaud in \cite{Mar}. The normalization positive constant $C_s=\frac{s}{\Gamma(1-s)}$, and $\Gamma$ denotes the Gamma function.
The nolocal operator $\mathcal{L}$ is of the form
\begin{equation}\label{eq:L}
\mathcal{L}  u(t,x) = \text{P.V.}\int_{\mathbb{R}^d}(u(t,x)-u(t,y)){K}(t,x,y) \operatorname{d}\!y,
\end{equation}
where $ {K}$ is a nonnegative measurable kernel satisfying the symmetry assumption
\begin{equation}\label{symm}
  {K}(t,x, y) = {K}(t,y, x)
\end{equation}
and the ellipticity condition
\begin{equation}\label{ellip}
  \Lambda^{-1}|x - y|^{-d-2s} \leq  {K}(t,x, y) \leq \Lambda|x - y|^{-d-2s}
\end{equation}
for a.e. $(t, x, y) \in \mathbb{R} \times \mathbb{R}^d \times \mathbb{R}^d$ with some $\Lambda\geq 1$. This integral operator obviously includes the fractional Laplacian operator $(-\Delta)^s$ when the kernel $K(t,x, y)=C_{n,s}|x - y|^{-d-2s}$. The space-time nonlocal equation in \eqref{eq1} serves as  a  prototypical model in the continuous time random walks \cite{MK}, which has diverse applications across physical and biological phenomena, refer to \cite{CM23} for detailed descriptions.
Observe that the definition of the fractional time derivative $\partial_{t}^{\alpha}$ given in \eqref{fratime} looks similar to the one-dimensional fractional Laplacian except that such integral only takes account of the interactions in the past.
This characteristic is attributed to the fact that the transport is irreversible in time and possesses a memory effect.

In the context of the nonlinear partial differential equations, numerous significant regularity results are derived from H\"{o}lder estimates for linear equations with only measurable coefficients.
The pioneering work on this subject dates back to
De Giorgi \cite{DG57}, who proved a local $C^\gamma$ estimate of solutions for second-order homogeneous uniformly elliptic equations in divergence form with measurable coefficients. Alternatively, Nash \cite{N58} independently obtained this H\"{o}lder continuity and established for the parabolic case.
Subsequently, Moser \cite{M64} provided a completely different proof of their results. Nowadays such fundamental regularity methods are known as De Giorgi-Nash-Moser theory, which has been extensively applied to various second-order nonlinear elliptic and parabolic equations, as detailed in the survey by Mingione \cite{M06}.

During the past two decades, the regularity theory for integro-differential equations with merely measurable kernels has attracted widespread attention not only in practical applications, but also from the perspective of pure mathematics. For the linear fractional elliptic equation, Silvestre \cite{S06} and Kassmann \cite{K09} proved local H\"{o}lder continuity of bounded weak solutions by extending the De Giorgi approach and the Moser iteration scheme, respectively.
To remove the priori boundedness assumption of weak solutions, Di Castro, Kuusi, and Palatucci \cite{DKP16} first introduced a nonlocal tail and derived the local H\"{o}lder continuity for the fractional $p$-Laplace elliptic equation based on the De Giorgi theory. For more H\"{o}lder regularity results for diverse types of nonlocal elliptic problems, please refer to \cite{BKO23, C17, KKP16, KuMinSi} and references therein.
With respect to the H\"{o}lder regularity results for fractional parabolic equations with the local time derivative $\partial_t u(x, t)$, comprehensive investigations have been carried out utilizing the modified De Giorgi method (cf.  \cite{BK24, CCV11, DZZ21, GK22, KW23, L24, NNSW23}). On the other hand, the relevant H\"{o}lder estimates were obtained in \cite{FK13, KW22} via the Moser technique.

In the pure time fractional case, Zacher et al. \cite{KRZ24, Z13} established the H\"{o}lder continuity of bounded solutions to local diffusion equations involving the Caputo derivative, with the results depending on the initial and boundary conditions. Regarding the space-time nonlocal parabolic equation analogous to type \eqref{eq1}, Allen, Caffarelli and Vasseur \cite{ACV16} showed a H\"{o}lder continuity of global weak solutions when the nonhomogeneous term $f$ is bounded.
In contrast, the local H\"{o}lder estimate for weak solutions to \eqref{eq1} seems to be rather incomplete, particularly in the absence of a priori boundedness assumption on weak solutions or without initial and boundary data.

Based on the De Giorgi approach modified to the space-time nonlocal setting, this paper aims to establish the local boundedness
and the local H\"{o}lder regularity of weak solutions for the nonhomogeneous fractional parabolic equations \eqref{eq1} with merely measurable kernel.
The main difficulty arises in the nonlocal nature of both the fractional time derivative and the fractional diffusion in the equation \eqref{eq1}, so it is necessary to take into account long range interactions both in time and space.
In order to obtain the desired local estimates, we highlight two quantities that will play crucial roles in the present extended local theory. One of them is the nonlocal tail term related to space, given by
\begin{equation}\label{Tail}
\operatorname{Tail}(u(t);x_0,r) = r^{2s} \int_{\R^d \setminus B_r(x_0)} \frac{|u(t,y)|}{|x_0-y|^{d+2s}} \operatorname{d}\! y.
\end{equation}
This quantity was originally detailed in \cite{DKP16}.
Furthermore,
the memory strongly affects the estimates,
so we also need to introduce a novel nonlocal tail term associated with time, defined as
\begin{equation}\label{tail}
\operatorname{tail}(u(x);t_0,r) = r^{2s} \int_{-\infty}^{t_0-r^2} \frac{|u(\tau,x)|}{(t_0-\tau)^{1+s}} \operatorname{d}\! \tau.
\end{equation}

To illustrate the main results of this paper, we start by presenting the definition of local weak solutions to the equation \eqref{eq1},
which utilizes the integration by parts formula for the Marchaud fractional derivative mentioned in \cite{A19}.
\begin{definition} Let $f\in L_{\rm loc}^2(I; L_{\rm loc}^2(\Omega))$. A function
$$u\in L_{\rm loc}^{\infty}(I;L_{\rm loc}^2(\Omega))\cap H_{\rm loc}^{\frac{s}{2}}(I;L_{\rm loc}^2(\Omega))\cap {\mathcal L}^-_{s}(\mathbb{R};L_{\rm loc}^\infty(\Omega))\cap L_{\rm loc}^2(I;H_{\rm loc}^{s}(\Omega))\cap L_{\rm loc}^\infty(I;{\mathcal L}^1_{2s}(\mathbb{R}^d))$$
 is called a local weak subsolution to \eqref{eq1}, if for any nonnegative test functions $\varphi\in L^2(I; H^s(\Omega)) \cap C^1(I;L^2(\Omega))$ with compact support
 $\varphi \equiv 0$ in $\left((t_1,t_2) \times \mathbb{R}^d\setminus \Omega'\right) \cup\left( (-\infty,t_1] \times \Omega' \right)$, where $(t_1,t_2)\times\Omega' \Subset I\times\Omega$,
	we have
	\begin{eqnarray*}
		&& -\int_{t_1}^{t_2} \int_{\Omega'} u(t,x)\partial^s_t\varphi(t,x) \operatorname{d}\! x \operatorname{d}\!t +\int_{t_1}^{t_2} \int_{\Omega'} \frac{u(t,x)\varphi(t,x)}{(t_2-t)^s} \operatorname{d}\!x\operatorname{d}\!t\\
&&+C_s
\int_{\Omega'} \int_{t_1}^{t_2}\int_{-\infty}^{t} \frac{(u(t,x)-u(\tau,x))(\varphi(t,x)-\varphi(\tau,x))}{(t-\tau)^{1+s}}\operatorname{d}\!\tau
\operatorname{d}\!t\operatorname{d}\!x\\
&& + \int_{t_1}^{t_2} \int_{\mathbb{R}^d} \int_{\mathbb{R}^d} (u(t,x)-u(t,y))(\varphi(t,x)-\varphi(t,y)) K(t,x,y) \operatorname{d}\!y \operatorname{d}\!x \operatorname{d}\! t \\
&\leq&\int_{t_1}^{t_2} \int_{\Omega'}f\varphi\operatorname{d}\!x \operatorname{d}\! t  .
\end{eqnarray*}
Here, $H^s$ and $H^{\frac{s}{2}}$ denote the fractional Sobolev spaces $W^{s,2}$ and $W^{\frac{s}{2},2}$ respectively.
${\mathcal L}^1_{2s}(\mathbb{R}^d)$ represents the nonlocal tail space defined as
	$${\mathcal L}^1_{2s}(\mathbb{R}^d):= \left \{v \in L^1_{\operatorname{loc}}(\mathbb{R}^d) \mathrel{\Big|} \int_{\mathbb{R}^d} \frac{|v(y)|}{1+|y|^{d+2s}} \operatorname{d}\!y < \infty \right \},$$
and ${\mathcal L}^-_{s}(\mathbb{R})$ is a class of slowly increasing functions given by
$$ {\mathcal L}^-_{s}(\mathbb{R}):=\{v\in L^1_{\rm loc} (\mathbb{R}) \mid \int_{-\infty}^t \frac{|v(\tau)|}{1+|\tau|^{1+s}}\operatorname{d}\!\tau<+\infty\,\, \mbox{for any} \,\, t\in I\}.$$
A function $u$ is called a local weak supersolution if the previous formula holds for every nonpositive test function $\varphi$. If a function $u$ is both a local weak subsolution and a local weak supersolution, then it is referred to as a local weak solution.
\end{definition}

In what follows, we denote the parabolic cylinder by
$Q_r(t_0,x_0):= I_r^{\ominus}(t_0) \times B_r(x_0)$, where
$I_r^{\ominus}(t_0) := (t_0-r^{2},t_0)$ and $B_r(x_0)$ represents an open ball in $\R^d$ with center $x_0$ and radius $r>0$. Now we are in a position to present the main results of this paper. The first one describes the local boundedness estimate of weak solutions to \eqref{eq1} in the a priori way.

\begin{theorem}
\label{lemma:locbd}
Let $u$ be a local weak solution to \eqref{eq1} in $I\times\Omega$ with
$f\in L_{\rm loc}^\infty(I;L_{\rm loc}^{\frac{d}{s}}(\Omega))$,
then for any $R>0$ and $(t_0,x_0)\in I\times\Omega$ with $I^\ominus_{2R}(t_0)\times B_{2R}(x_0)\Subset I\times\Omega$, there exists a positive constant $C$, depending only on $d,\,s,\,\Lambda$, such that
\begin{eqnarray}\label{locbd}
  \sup_{Q_{\frac{R}{2}}(t_0,x_0)} |u| &\leq& C\left( \fint_{Q_R(t_0,x_0)} |u(t,x)| \operatorname{d}\!x\operatorname{d}\!t +  \sup_{t\in I_R^{\ominus}(t_0)} \operatorname{Tail}(u(t);x_0,R)+ \sup_{x\in B_R(x_0)} \operatorname{tail}(u(x);t_0,R)\right.\nonumber\\
  &&\left.+R^s\|f\|_{L^\infty(I_R^{\ominus}(t_0);L^{\frac{d}{s}}(B_R(x_0)))}\right).
\end{eqnarray}
\end{theorem}

By virtue of the aforementioned $L^\infty-L^1$ estimate with tails, and together with a  growth lemma developed below, we can further derive the local parabolic H\"{o}lder continuity of weak solutions to \eqref{eq1}.

\begin{theorem}
	\label{prop:Holderq}
 Let $d>2$ and $u$ be a local weak solution to \eqref{eq1} in $I\times\Omega$ with
$f\in L_{\rm loc}^\infty(I;L_{\rm loc}^{\frac{d}{s}}(\Omega))$,
then $u \in C^{\frac{\gamma}{2},\gamma}_{\rm loc}(I\times\Omega)$ for some $\gamma=\gamma(d,s,\Lambda) \in (0,1)$, and satisfies
	\begin{eqnarray*}
   [ u ]_{C^{\frac{\gamma}{2},\gamma}(Q_{R}(t_0,x_0))}&\leq& C R^{-\gamma}  \fint_{Q_{2R}(t_0,x_0)} |u(t,x)|\operatorname{d}\! x \operatorname{d}\! t   \\
    &&  + CR^{-\gamma}  \left(  \sup_{t\in I_{2R}^{\ominus}(t_0)} \operatorname{Tail}(u(t);x_0,2R)+ \sup_{x\in B_{2R}(x_0)} \operatorname{tail}(u(x);t_0,2R)\right)\\
    &&+CR^{s-\gamma}\|f\|_{L^\infty(I_{2R}^{\ominus}(t_0);L^{\frac{d}{s}}(B_{2R}(x_0)))}
 \end{eqnarray*}
for any $R>0$ and $(t_0,x_0)\in I\times\Omega$ with $(t_0-(2R)^2,t_0+(2R)^2)\times B_{2R}(x_0)\Subset I\times\Omega$, where
\begin{equation*}
  [ u ]_{C^{\frac{\gamma}{2},\gamma}(Q_{R}(t_0,x_0))}:=\sup_{\substack{(t,x),(\tau,y)\in Q_{R}(t_0,x_0)\\ (t,x)\neq(\tau,y) } }\frac{|u(t,x)-u(\tau,y)|}{\left(|t-\tau|^{\frac{1}{2}}+|x-y|\right)^\gamma},
\end{equation*}
and the positive constant $C$ depends only on $d,\,s,\,\Lambda$.
\end{theorem}

\begin{remark}
Note that the $L^\infty$ tails
$$ \displaystyle\sup_{t\in I_r^{\ominus}(t_0)} \operatorname{Tail}(u(t);x_0,r) \,\,\, \mbox{and} \,\,\displaystyle\sup_{x\in B_r(x_0)} \operatorname{tail}(u(x);t_0,r)$$ are finite whenever $u$ belongs to ${\mathcal L}^-_{s}(\mathbb{R};L^\infty(B_r(x_0)))$ and $L^\infty(I_r^{\ominus}(t_0);{\mathcal L}^1_{2s}(\mathbb{R}^d))$ respectively.
Analogous outcomes with more natural $L^2$-Tail have recently been proved in \cite{KW23, NNSW23} for parabolic equations involving the standard local time derivative, the treatment of local time derivatives in the aforementioned studies is not applicable to the parabolic equations considered in this paper, which possess derivatives with respect to time of fractional order.
Although we do not provide more natural assumptions on the tail terms, our results are still novel for the fully fractional parabolic equation.
\end{remark}

\begin{remark}
The previous results are new even for the fully fractional parabolic equation \eqref{eq1} with $f\equiv 0$ or a bounded nonhomogeneous term. In both scenarios, the condition $d>2$ in Theorem \ref{prop:Holderq} can be relaxed to $d\geq2$. Furthermore,
with only marginal modifications to the proof, the main results remain valid for the local weak solution
$$u\in L_{\rm loc}^{\infty}(I;L_{\rm loc}^2(\Omega))\cap H_{\rm loc}^{\frac{\alpha}{2}}(I;L_{\rm loc}^2(\Omega))\cap {\mathcal L}^-_{\alpha}(\mathbb{R};L_{\rm loc}^\infty(\Omega))\cap L_{\rm loc}^2(I;H_{\rm loc}^{s}(\Omega))\cap L_{\rm loc}^\infty(I;{\mathcal L}^1_{2s}(\mathbb{R}^d))$$
of
\begin{equation*}
\partial^\alpha_t u(t,x) +\mathcal{L} u(t,x) = f(t,x) ~~\text{in}~~I \times \Omega,
\end{equation*}
where the orders are decoupled in time $\alpha\in (0,1)$ and in space $s\in(0,1)$.
More precisely, at this point, the parabolic cylinder given by
$$Q_r(t_0,x_0):= (t_0-r^{\frac{2s}{\alpha}},t_0) \times B_r(x_0),$$
the constants $C$ and $\gamma$ in Theorem \ref{lemma:locbd} and Theorem \ref{prop:Holderq} also depend on $\alpha$, and the definition of the tail term associated with time is replaced by
\begin{equation}\label{tail}
\operatorname{tail}(u(x);t_0,r) = r^{2s} \int_{-\infty}^{t_0-r^{\frac{2s}{\alpha}}} \frac{|u(\tau,x)|}{(t_0-\tau)^{1+\alpha}} \operatorname{d}\! \tau.
\end{equation}
Here we assume the orders $\alpha=s$ to make the proof more concise and clear.
\end{remark}

The remainder part of this paper proceeds as follows. Section \ref{section2} consists of some preliminary results and establishes the Caccioppoli inequality with tails. Section \ref{section3} is devoted to proving the local boundedness stated in Theorem \ref{lemma:locbd}. In section \ref{section4}, we finally complete the proof of Theorem \ref{prop:Holderq} regarding H\"{o}lder continuity.

\section{Preliminaries}\label{section2}

In this section, we collect several useful auxiliary results and establish the Caccioppoli-type inequality for the space-time nonlocal parabolic equation \eqref{eq1}. These results are crucial for the proof of our main theorems.

\subsection{Auxiliary results}

Let us start with two classical fractional Sobolev inequalities established in \cite{DPV12}, where the second embedding theorem can be derived by using scaling argument.

\begin{lemma}\label{fracsobolev} {\rm(cf. \cite[Theorem 6.5]{DPV12})}
Let $s\in(0,1)$ and $n,\,p\in[1,+\infty)$ be such that $sp<n$. Then
for any compactly supported function $f\in W^{s,p} (\R^n)$, there holds that
\begin{equation*}
  \|f\|_{L^{\frac{np}{n-sp}}(\R^n)}^p\leq C\int_{\R^n}\int_{\R^n}\frac{|f(x)-f(y)|^p}{|x-y|^{n+sp}}\operatorname{d}\! x \operatorname{d}\! y,
\end{equation*}
where the positive constant $C$ only depending on $n,\,p$ and $s$.
\end{lemma}

\begin{lemma}\label{localfracsobolev}  {\rm(cf. \cite[Theorem 6.7]{DPV12})}
Let $s\in(0,1)$ and $p\in[1,+\infty)$ satisfy $sp<d$. Then for any $f\in W^{s,p} (B_r)$ with $B_r\subset \R^d$,
there exists a positive constant $C=C(d,p,s)$ such that
\begin{equation*}
  \|f\|_{L^{\frac{np}{n-sp}}(B_r)}^p\leq C\left(r^{-sp}\int_{B_r}|f(x)|^p\operatorname{d}\! x+\int_{B_r}\int_{B_r}\frac{|f(x)-f(y)|^p}{|x-y|^{d+sp}}\operatorname{d}\! x \operatorname{d}\! y\right).
\end{equation*}
\end{lemma}

The following energy estimate plays a pivotal role in dealing with the Marchaud fractional derivative when establishing the Caccioppoli-type inequality.
\begin{lemma}\label{extendt} {\rm(cf. \cite[Lemma 2.2]{ACV16})}
Let $f\in C([a,T])$. If we extend $f$ to all of $\R$ by defining $f(t)=0$ for $t<a$ and then reflecting evenly across $T$, then there exists a positive constant $C=C(s)$ such that
\begin{equation*}
  \int_{\R}\int_{\R}\frac{|f(t)-f(\tau)|^2}{|t-\tau|^{1+s}}\operatorname{d}\! \tau \operatorname{d}\! t\leq C\left(\int_{a}^T\int_{a}^t\frac{|f(t)-f(\tau)|^2}{|t-\tau|^{1+s}}\operatorname{d}\! \tau \operatorname{d}\! t+\int_{a}^T\frac{|f(t)|^2}{(T-t)^{s}}\operatorname{d}\! t\right).
\end{equation*}
\end{lemma}

We end this subsection by presenting two lemmas concerning the geometric convergence of sequences of numbers, along with an iterative lemma, which are essential in
carrying out the De Giorgi theory in the context of space-time nonlocal settings.
\begin{lemma}\label{iterate1} {\rm(cf. \cite[Chapter1, Lemma 4.1]{Dib93})}
Let $\{A_i\}$ be a sequence of positive numbers such that
$$A_{i+1}\leq Cb^iA_i^{1+\beta},$$
where $C,\,b>1$ and $\beta>0$ are given numbers. If $A_0\leq C^{-\frac{1}{\beta}}b^{-\frac{1}{\beta^2}}$, then $\{A_i\}$ converges to zero as $i\rightarrow\infty$.
\end{lemma}

\begin{lemma}\label{iterate2} {\rm(cf. \cite[Chapter1, Lemma 4.2]{Dib93})}
Let $\{A_i\}$ and $\{B_i\}$ be sequences of positive numbers, satisfying the recursive inequalities
\begin{equation*}
  \left\{\begin{array}{r@{\ \ }c@{\ \ }ll}
  A_{i+1}&\leq& Cb^i(A_i^{1+\beta}+A_i^{\beta}B_i^{1+\gamma}),\\
  B_{i+1}&\leq& Cb^i(A_i+B_i^{1+\gamma}),
  \end{array}\right.
\end{equation*}
where $C,\,b>1$ and $\beta,\,\gamma>0$ are given numbers.  If
$$A_0+B_0^{1+\gamma}\leq (2C)^{-\frac{1+\gamma}{\sigma}}b^{-\frac{1+\gamma}{\sigma^2}}$$
with $\sigma=\min\{\beta,\gamma\}$, then
$\{A_i\}$ and $\{B_i\}$ tend to zero as $i\rightarrow\infty$.
\end{lemma}

\begin{lemma}\label{iterate3} {\rm(cf. \cite[Lemma 6.1]{Giu})}
Let $f$ be a bounded nonnegative function defined on $[r,R]$.
Assume that for $ r\leq \rho_1<\rho_2\leq R$ we have
\begin{equation*}
  f(\rho_1)\leq \vartheta f(\rho_2)+\frac{c_1}{(\rho_2-\rho_1)^{\beta}}+c_2,
\end{equation*}
where $c_1,\,c_2\geq 0$, $\beta>0$, and $0\leq\vartheta<1$. Then there exists a positive constant $C=C(\beta,\vartheta)$
such that
\begin{equation*}
  f(r)\leq C\left[\frac{c_1}{(R-r)^{\beta}}+c_2\right].
\end{equation*}
\end{lemma}

\subsection{Caccioppoli inequality}

In this subsection, we focus on establishing a Caccioppoli-type inequality with tails for the space-time nonlocal parabolic equation \eqref{eq1}.
When there is no confusion, we usually omit the vertex $(x_0,t_0)$ from the parabolic cylinder $Q_{R}(x_0,t_0)$ for convenient.
In what follows, $C$ denotes a constant whose value may be
different from line to line, and only the relevant dependence is specified in parentheses.

\begin{lemma}
\label{lemma:Cacc}
Let $u$ be a local weak subsolution to \eqref{eq1} in $I\times\Omega$,
then for any $R>0$ and $(t_0,x_0)\in I\times\Omega$ with $I^\ominus_{2R}(t_0)\times B_{2R}(x_0)\Subset I\times\Omega$, there exists a positive constant $C$ depending only on $d,\,s,\,\Lambda$ such that
\begin{eqnarray}\label{eq:Cacc}
 &&  \int_{I_{r}^{\ominus}(t_0)}\int_{B_r(x_0)} \int_{B_r(x_0)}\frac{|w_{+}(t,x)-w_{+}(t,y)|^2}{|x-y|^{d+2s}}\operatorname{d}\! x \operatorname{d}\! y\operatorname{d}\!t \nonumber\\
 &&+\int_{B_r(x_0)}\int_{I_{r}^{\ominus}(t_0)} \int_{I_{r}^{\ominus}(t_0)}\frac{|w_{+}(t,x)-w_{+}(\tau,x)|^2}{|t-\tau|^{1+s}}\operatorname{d}\! \tau \operatorname{d}\! t\operatorname{d}\!x \nonumber \\
 &&+ r^{-d-2s} \int_{I_{r}^{\ominus}(t_0)} \int_{B_r(x_0)} \int_{B_r(x_0)} w_{+}(t,x) w_{-}(t,y) \operatorname{d}\! x \operatorname{d}\! y \operatorname{d}\! t\nonumber\\
 &\leq&  C \left[\left(\frac{r+\rho_1}{\rho_1}\right)^2\rho_1^{-2s}\vee\left(\frac{r+\rho_2}{\rho_2}\right)^2\rho_2^{-2s}\right]\int_{I_{r + \rho_1}^{\ominus}(t_0)} \int_{B_{r+\rho_2}(x_0)} w^2_{+}(t,x) \operatorname{d}\! x \operatorname{d}\! t \nonumber\\
 &&+ C\left(\frac{r+\rho_1}{\rho_1}\right)^2 \rho_1^{-2s} \int_{B_{r+\rho_2}(x_0)} \left(\int_{I_{r+\rho_1}^{\ominus}(t_0)} w_{+}(t,x) \operatorname{d}\!t \right) \operatorname{tail}(w_{+}(x);t_0,r+\rho_1)\operatorname{d}\!x\nonumber\\
&&+ C\left(\frac{r+\rho_2}{\rho_2}\right)^d \rho_2^{-2s} \int_{I_{r+\rho_1}^{\ominus}(t_0)} \left(\int_{B_{r+\rho_2}(x_0)} w_{+}(t,x) \operatorname{d}\!x \right) \operatorname{Tail}(w_{+}(t);x_0,r+\rho_2)\operatorname{d}\!t\nonumber\\
&&+C\int_{I_{r+\rho_1}^{\ominus}(t_0)} \int_{B_{r+\rho_2}(x_0)}|f(t,x)|w_{+}(t,x)\operatorname{d}\!x \operatorname{d}\! t
\end{eqnarray}
for any $0 < \rho_1\leq r \leq r+\rho_1 \leq R$ and $0 < \rho_2 \leq r \leq r+\rho_2 \leq R$ such that $I_{r+\rho_1}^{\ominus}(t_0) \times B_{r+\rho_2}(x_0) \subset Q_{2R}(t_0,x_0)$, where $w = u-l$ with a level $l \in \R$, the lower truncation $w_+:=\max\{u-l,0\}$, the upper truncation $w_-:=-\min\{u-l,0\}$, and $a\vee b:=\max\{a,b\}$.
\end{lemma}

\begin{remark}
The Caccioppoli-type inequality \eqref{eq:Cacc} still holds for $w_-$ when $u$ is a local weak supersolution to \eqref{eq1} in $I\times\Omega$. The set of functions satisfying Caccioppoli inequality is typically called a De Giorgi class.
\end{remark}

\begin{proof}
The proof goes by testing the weak formula with the test function $\varphi(t,x) = \eta^2(t)\phi^2(x)w_+(t,x)$, where the time truncation function  $\eta \in C^{\infty}((-\infty,t_0);[0,1])$ with $\eta \equiv 1$ in $I_r^\ominus$, $\eta \equiv 0$ in $(-\infty,t_0-(r+\frac{\rho_1}{2})^2]$, and $|\eta'| \leq C \left((r+\rho_1)^2-r^2\right)^{-1}$, and
the space cut-off function $\phi \in C_0^{\infty}(B_{r+\frac{\rho_2}{2}})$ with $\phi \equiv 1$ in $B_{r}$, $0 \leq \phi\leq 1$, and $|\nabla \phi| \leq C \rho_2^{-1}$. As pointed out in \cite[Lemma 2.3]{A19}, similar to the case of $s=1$, the Marchaud fractional derivative retains the property of commuting with the Steklov averaging operator, which is the main advantage to considering this type of fractional time derivative.
Thus, although the selected test function relies on $u$, which lacks the required regularity, it can be rigorously justified by using Steklov averages technique as in \cite[Section 7]{KW22b} and \cite[Section 2]{A19}.

Note the fact that
$u(t,x)=w_+(t,x)-w_-(t,x)+l$ and
$$-w_+(t,x)\partial^s_tw_-(t,x)\geq 0,$$
we have the following weak subsolution formula
\begin{eqnarray}\label{weaksol1}
		&&\int_{I_{r+\rho_1}^{\ominus}} \int_{B_{r+\rho_2}}f\varphi\operatorname{d}\!x \operatorname{d}\! t\nonumber \\ &\geq&-\int_{I_{r+\rho_1}^{\ominus}} \int_{B_{r+\rho_2}} w_+(t,x)\partial^s_t\varphi(t,x) \operatorname{d}\! x \operatorname{d}\!t +\int_{I_{r+\rho_1}^{\ominus} }\int_{B_{r+\rho_2}} \frac{w_+(t,x)\varphi(t,x)}{(t_0-t)^s} \operatorname{d}\!x\operatorname{d}\!t\nonumber\\
&&+C_s
\int_{B_{r+\rho_2}} \int_{I_{r+\rho_1}^{\ominus}}\int_{-\infty}^{t} \frac{(w_+(t,x)-w_+(\tau,x))(\varphi(t,x)-\varphi(\tau,x))}{(t-\tau)^{1+s}}\operatorname{d}\!\tau
\operatorname{d}\!t\operatorname{d}\!x\nonumber\\
&& + \int_{I_{r+\rho_1}^{\ominus}} \int_{\mathbb{R}^d} \int_{\mathbb{R}^d} (u(t,x)-u(t,y))(\varphi(t,x)-\varphi(t,y)) K(t,x,y) \operatorname{d}\!y \operatorname{d}\!x \operatorname{d}\! t.
\end{eqnarray}
We directly calculate the first term on the right side of \eqref{weaksol1}, and apply the integration by parts formula for the Marchaud derivative to derive
\begin{eqnarray*}
   &&-\int_{I_{r+\rho_1}^{\ominus}} \int_{B_{r+\rho_2}} w_+(t,x)\partial^s_t\varphi(t,x) \operatorname{d}\! x \operatorname{d}\!t\\
   &=& -\int_{I_{r+\rho_1}^{\ominus}} \int_{B_{r+\rho_2}} \phi^2(x)\eta^2(t)\left(w_+(t,x)\partial^s_tw_+(t,x)\right) \operatorname{d}\! x \operatorname{d}\!t \\
  &&- C_s\int_{I_{r+\rho_1}^{\ominus}} \int_{B_{r+\rho_2}} \int_{-\infty}^t \frac{(\eta^2(t)-\eta^2(\tau))w_+(\tau,x) w_+(t,x)\phi^2(x)}{(t-\tau)^{1+s}}
  \operatorname{d}\! \tau  \operatorname{d}\! x \operatorname{d}\!t\\
   &=&-\frac{1}{2}\int_{B_{r+\rho_2}}   \phi^2(x)\left(\int_{I_{r+\rho_1}^{\ominus}} \eta^2(t)\partial^s_tw^2_+(t,x) \operatorname{d}\!t\right)\operatorname{d}\! x \\
   &&-\frac{C_s}{2}\int_{I_{r+\rho_1}^{\ominus}} \int_{B_{r+\rho_2}} \int_{-\infty}^t\frac{(w_+(t,x)-w_+(\tau,x))^2 \phi^2(x)\eta^2(t)}{(t-\tau)^{1+s}}
  \operatorname{d}\! \tau  \operatorname{d}\! x \operatorname{d}\!t\\
  &&- C_s\int_{I_{r+\rho_1}^{\ominus}} \int_{B_{r+\rho_2}} \int_{-\infty}^t\frac{(\eta^2(t)-\eta^2(\tau))w_+(\tau,x) w_+(t,x)\phi^2(x)}{(t-\tau)^{1+s}}
  \operatorname{d}\! \tau  \operatorname{d}\! x \operatorname{d}\!t\\
   &=&\frac{1}{2}\int_{B_{r+\rho_2}}   \phi^2(x)\left[\int_{I_{r+\rho_1}^{\ominus}} w^2_+(t,x)\partial^s_t\eta^2(t) \operatorname{d}\!t-\int_{I_{r+\rho_1}^{\ominus}} \frac{w^2_+(t,x)\eta^2(t)}{(t_0-t)^s} \operatorname{d}\!t\right]\operatorname{d}\! x\\
   &&-\frac{C_s}{2}\int_{I_{r+\rho_1}^{\ominus}} \int_{B_{r+\rho_2}} \int_{-\infty}^t\frac{(w^2_+(t,x)-w^2_+(\tau,x))(\eta^2(t)-\eta^2(\tau))\phi^2(x)}{(t-\tau)^{1+s}}
  \operatorname{d}\! \tau  \operatorname{d}\! x \operatorname{d}\!t\\
   &&-\frac{C_s}{2} \int_{B_{r+\rho_2}}\phi^2(x) \int_{I_{r+\rho_1}^{\ominus}} \int_{-\infty}^t\frac{(w_+(t,x)-w_+(\tau,x))^2\eta^2(t)}{(t-\tau)^{1+s}}
  \operatorname{d}\! \tau  \operatorname{d}\!t\operatorname{d}\! x\\
  &&- C_s\int_{B_{r+\rho_2}}\phi^2(x) \int_{I_{r+\rho_1}^{\ominus}} \int_{-\infty}^t\frac{(\eta^2(t)-\eta^2(\tau))w_+(\tau,x) w_+(t,x)}{(t-\tau)^{1+s}}
  \operatorname{d}\! \tau  \operatorname{d}\!t\operatorname{d}\! x
\end{eqnarray*}
Substituting the above equality into \eqref{weaksol1}and utilizing the nondecreasing and nonnegative properties of $\eta$, we
show that the weak  subsolution $u$ satisfies
\begin{eqnarray}\label{weaksol2}
&&\int_{I_{r+\rho_1}^{\ominus}(t_0)} \int_{B_{r+\rho_2}(x_0)}|f(t,x)|w_{+}(t,x)\operatorname{d}\!x \operatorname{d}\! t\nonumber\\
		 &\geq&\int_{I_{r+\rho_1}^{\ominus}} \int_{B_{r+\rho_2}}f\varphi\operatorname{d}\!x \operatorname{d}\! t\nonumber \\ &\geq&\frac{1}{2}\int_{I_{r+\rho_1}^{\ominus} }\int_{B_{r+\rho_2}} \frac{w^2_+(t,x)\eta^2(t)\phi^2(x)}{(t_0-t)^s} \operatorname{d}\!x\operatorname{d}\!t\nonumber\\
&&+\frac{C_s}{2} \int_{B_{r+\rho_2}}\phi^2(x) \int_{I_{r+\rho_1}^{\ominus}} \int_{-\infty}^t\frac{(w_+(t,x)\eta(t)-w_+(\tau,x)\eta(\tau))^2}{(t-\tau)^{1+s}}
  \operatorname{d}\! \tau  \operatorname{d}\!t\operatorname{d}\! x\nonumber\\
  &&- C_s\int_{B_{r+\rho_2}}\phi^2(x) \int_{I_{r+\rho_1}^{\ominus}} \int_{-\infty}^t\frac{(\eta^2(t)-\eta^2(\tau))w_+(\tau,x) w_+(t,x)}{(t-\tau)^{1+s}}
  \operatorname{d}\! \tau  \operatorname{d}\!t\operatorname{d}\! x\nonumber\\
&& + \int_{I_{r+\rho_1}^{\ominus}} \int_{\mathbb{R}^d} \int_{\mathbb{R}^d} (u(t,x)-u(t,y))(\varphi(t,x)-\varphi(t,y)) K(t,x,y) \operatorname{d}\!y \operatorname{d}\!x \operatorname{d}\! t\nonumber\\
&:=&I+II+III+IV.
\end{eqnarray}

Based on the definition of $\phi$, we first estimate $I$ and $II$ as follows
\begin{equation*}
  I\geq \frac{1}{2}\int_{B_{r}} \int_{I_{r+\rho_1}^{\ominus} } \frac{w^2_+(t,x)\eta^2(t)}{(t_0-t)^s} \operatorname{d}\!x\operatorname{d}\!t,
\end{equation*}
and
\begin{equation*}
  II\geq\frac{C_s}{2}\int_{B_r}\int_{I_{r+\rho_1}^{\ominus}} \int_{t_0-(r+\rho_1)^2}^t\frac{(w_+(t,x)\eta(t)-w_+(\tau,x)\eta(\tau))^2}{(t-\tau)^{1+s}}\operatorname{d}\! \tau \operatorname{d}\! t\operatorname{d}\!x.
\end{equation*}
By virtue of Lemma \ref{extendt}, we further
extend $w_+\eta$ evenly in $t$ across the entire $\R$ by
maintaining symmetry with respect to $t=t_0$, we deduce that
\begin{eqnarray*}
&&\int_{B_{r}} \int_{I_{r}^{\ominus}} \int_{I_{r}^{\ominus}}\frac{|w_{+}(t,x)-w_{+}(\tau,x)|^2}{|t-\tau|^{1+s}}\operatorname{d}\! \tau \operatorname{d}\! t\operatorname{d}\!x\nonumber\\
&\leq&\int_{B_{r}} \int_{\R} \int_{\R}\frac{\left(w_{+}(t,x)\eta(t)-w_{+}(\tau,x)\eta(\tau)\right)^2}{|t-\tau|^{1+s}}\operatorname{d}\! \tau \operatorname{d}\! t\operatorname{d}\!x\nonumber\\
&\leq& C(s)\left(\int_{B_r}\int_{I_{r+\rho_1}^{\ominus}} \int_{t_0-(r+\rho_1)^2}^t\frac{\left(w_{+}(t,x)\eta(t)-w_{+}(\tau,x)\eta(\tau)\right)^2}{(t-\tau)^{1+s}}\operatorname{d}\! \tau \operatorname{d}\! t\operatorname{d}\!x
  +\int_{B_{r}} \int_{I_{r+\rho_1}^{\ominus} }\frac{w^2_+(t,x)\eta^2(t)}{(t_0-t)^s} \operatorname{d}\!t \operatorname{d}\!x\right)\\
  &\leq&C(s)(I+II).
\end{eqnarray*}

Next, we estimate $III$ by decomposing the integral region into two parts
\begin{eqnarray*}
  -III &=& C_s\int_{B_{r+\rho_2}}\phi^2(x) \int_{I_{r+\rho_1}^{\ominus}} \int_{t_0-(r+\rho_1)^2}^t\frac{(\eta^2(t)-\eta^2(\tau))w_+(\tau,x) w_+(t,x)}{(t-\tau)^{1+s}}
  \operatorname{d}\! \tau  \operatorname{d}\!t\operatorname{d}\! x \\
   &&+ C_s\int_{B_{r+\rho_2}}\phi^2(x) \int_{I_{r+\frac{\rho_1}{2}}^{\ominus}} \int_{-\infty}^{t_0-(r+\rho_1)^2}\frac{\eta^2(t)w_+(\tau,x) w_+(t,x)}{(t-\tau)^{1+s}}
  \operatorname{d}\! \tau  \operatorname{d}\!t\operatorname{d}\! x \\
  &:=&III_1+III_2.
\end{eqnarray*}
With respect to the estimate of $III_1$, it follows from the choice of $\eta$ that
\begin{eqnarray*}
  III_1 &\leq& C\rho_1^{-2}\int_{B_{r+\rho_2}} \int_{I_{r+\rho_1}^{\ominus}} \int_{I_{r+\rho_1}^{\ominus}} \frac{w_+(\tau,x) w_+(t,x)}{|t-\tau|^{s}}
  \operatorname{d}\! \tau  \operatorname{d}\!t\operatorname{d}\! x \\
   &\leq&  C\rho_1^{-2}\int_{B_{r+\rho_2}} \int_{I_{r+\rho_1}^{\ominus}} w^2_+(t,x)\left( \int_{I_{r+\rho_1}^{\ominus}} \frac{1}{|t-\tau|^{s}}
  \operatorname{d}\! \tau\right)  \operatorname{d}\!t\operatorname{d}\! x \\
  &\leq&  C(s)(\frac{r+\rho_1}{\rho_1})^2\rho_1^{-2s}\int_{I_{r+\rho_1}^{\ominus}}\int_{B_{r+\rho_2}} w^2_+(t,x) \operatorname{d}\! x\operatorname{d}\!t\,.
\end{eqnarray*}
By performing a direct calculation for $III_2$, we obtain
\begin{eqnarray*}
  III_2 &\leq& C \int_{B_{r+\rho_2}} \int_{I_{r+\frac{\rho_1}{2}}^{\ominus}} w_+(t,x)\left( \int_{-\infty}^{t_0-(r+\rho_1)^2}\frac{w_+(\tau,x) }{(t-\tau)^{1+s}}
  \operatorname{d}\! \tau \right) \operatorname{d}\!t\operatorname{d}\! x\\
   &\leq&  C(s) \int_{B_{r+\rho_2}} \int_{I_{r+\frac{\rho_1}{2}}^{\ominus}} w_+(t,x)\left[(\frac{\rho_1}{r+\rho_1})^{-2(1+s)}\int_{-\infty}^{t_0-(r+\rho_1)^2}\frac{w_+(\tau,x) }{(t_0-\tau)^{1+s}}
  \operatorname{d}\! \tau \right] \operatorname{d}\!t\operatorname{d}\! x\\
     &\leq&  C(s) \left(\frac{r+\rho_1}{\rho_1}\right)^2 \rho_1^{-2s} \int_{B_{r+\rho_2}} \left(\int_{I_{r+\rho_1}^{\ominus}} w_{+}(t,x) \operatorname{d}\!t \right) \operatorname{tail}(w_{+}(x);t_0,r+\rho_1)\operatorname{d}\!x\,.
\end{eqnarray*}

A combination of dividing the integral region and utilizing the symmetry of the kernel $K$ allows us to estimate $IV$ as follows
\begin{eqnarray*}
  IV &=& \int_{I_{r+\rho_1}^{\ominus}} \int_{\mathbb{R}^d} \int_{\mathbb{R}^d} (u(t,x)-u(t,y))\left(w_+(t,x)\phi^2(x)-w_+(t,y)\phi^2(y)\right)\eta^2(t) K(t,x,y) \operatorname{d}\!y \operatorname{d}\!x \operatorname{d}\! t \\
   &=& \int_{I_{r+\rho_1}^{\ominus}} \int_{B_{r+\rho_2}} \int_{B_{r+\rho_2}} (u(t,x)-u(t,y))\left(w_+(t,x)\phi^2(x)-w_+(t,y)\phi^2(y)\right)\eta^2(t) K(t,x,y) \operatorname{d}\!y \operatorname{d}\!x \operatorname{d}\! t  \\
   &&+2\int_{I_{r+\rho_1}^{\ominus}} \int_{\mathbb{R}^d\setminus B_{r+\rho_2}} \int_{B_{r+\rho_2}} (u(t,y)-u(t,x))w_+(t,y)\phi^2(y)\eta^2(t) K(t,x,y) \operatorname{d}\!y \operatorname{d}\!x \operatorname{d}\! t\\
   &:=&IV_1+IV_2.
\end{eqnarray*}
Let the superlevel set
$$A_l:=\{u(\cdot,t)>l\}\cap B_{r+\rho_2}\,\, \mbox{for fixed}\,\, t\in I_{r+\rho_1}^{\ominus}.$$
If $x\in A_l$ and $y\in B_{r+\rho_2}\setminus A_l$, then
\begin{eqnarray*}
   && (u(t,x)-u(t,y))\left(w_+(t,x)\phi^2(x)-w_+(t,y)\phi^2(y)\right) \\
   &=& (u(t,x)-l+l-u(t,y))w_+(t,x)\phi^2(x)\\
   &=& (w_+(t,x)+w_-(t,y))w_+(t,x)\phi^2(x)\\
   &\geq& w_+(t,x)w_-(t,y)\phi^2(x).
\end{eqnarray*}
If $x,\,y\in A_l$, then
\begin{eqnarray*}
   &&(u(t,x)-u(t,y))\left(w_+(t,x)\phi^2(x)-w_+(t,y)\phi^2(y)\right) \\ &=&(w_+(t,x)-w_+(t,y))\left(w_+(t,x)\phi^2(x)-w_+(t,y)\phi^2(y)\right)\\
   &=&\left(w_+(t,x)\phi(x)-w_+(t,y)\phi(y)\right)^2-
   w_+(t,x)w_+(t,y)(\phi(x)-\phi(y))^2.
\end{eqnarray*}
While if $x,\,y\in B_{r+\rho_2}\setminus A_l$, then
$$(u(t,x)-u(t,y))\left(w_+(t,x)\phi^2(x)-w_+(t,y)\phi^2(y)\right)=0.$$
Taking into account the aforementioned estimates and combining the ellipticity condition \eqref{ellip} with the selection of $\phi$, we estimate $IV_1$ in the following manner
\begin{eqnarray*}
  IV_1 &\geq& \int_{I_{r+\rho_1}^{\ominus}} \int_{B_{r+\rho_2}} \int_{B_{r+\rho_2}} w_+(t,x)w_-(t,y)\phi^2(x)\eta^2(t) K(t,x,y) \operatorname{d}\!y \operatorname{d}\!x \operatorname{d}\! t \\
   && +\int_{I_{r+\rho_1}^{\ominus}} \int_{B_{r+\rho_2}} \int_{B_{r+\rho_2}} \left(w_+(t,x)\phi(x)-w_+(t,y)\phi(y)\right)^2\eta^2(t) K(t,x,y) \operatorname{d}\!y \operatorname{d}\!x \operatorname{d}\! t\\
&& -\int_{I_{r+\rho_1}^{\ominus}} \int_{B_{r+\rho_2}} \int_{B_{r+\rho_2}} w_+(t,x)w_+(t,y)(\phi(x)-\phi(y))^2\eta^2(t) K(t,x,y) \operatorname{d}\!y \operatorname{d}\!x \operatorname{d}\! t\\
   &\geq&\Lambda^{-1}\int_{I_{r+\rho_1}^{\ominus}} \int_{B_{r+\rho_2}} \int_{B_{r+\rho_2}} w_+(t,x)w_-(t,y)\phi^2(x)\eta^2(t) |x - y|^{-d-2s} \operatorname{d}\!y \operatorname{d}\!x \operatorname{d}\! t\\
   && +\Lambda^{-1}\int_{I_{r+\rho_1}^{\ominus}} \int_{B_{r+\rho_2}} \int_{B_{r+\rho_2}} \left(w_+(t,x)\phi(x)-w_+(t,y)\phi(y)\right)^2\eta^2(t) |x - y|^{-d-2s} \operatorname{d}\!y \operatorname{d}\!x \operatorname{d}\! t\\
   && -\Lambda\int_{I_{r+\rho_1}^{\ominus}} \int_{B_{r+\rho_2}} \int_{B_{r+\rho_2}} \left(w^2_+(t,x)+w^2_+(t,y)\right)|\nabla\phi|^2\eta^2(t)|x - y|^{-d-2s+2} \operatorname{d}\!y \operatorname{d}\!x \operatorname{d}\! t\\
     &\geq& C(\Lambda,d,s)r^{-d-2s} \int_{I_{r}^{\ominus}} \int_{B_r} \int_{B_r} w_{+}(t,x) w_{-}(t,y) \operatorname{d}\! x \operatorname{d}\! y \operatorname{d}\! t\\
   &&+\Lambda^{-1}\int_{I_{r}^{\ominus}} \int_{B_{r}} \int_{B_{r}} \frac{\left|w_+(t,x)-w_+(t,y)\right|^2}{|x - y|^{d+2s}} \operatorname{d}\!y \operatorname{d}\!x\operatorname{d}\! t\\
    && -C(\Lambda)\rho_2^{-2}\int_{I_{r+\rho_1}^{\ominus}} \int_{B_{r+\rho_2}} w^2_+(t,x)\left( \int_{B_{r+\rho_2}} |x - y|^{-d-2s+2} \operatorname{d}\!y \right) \operatorname{d}\!x \operatorname{d}\! t\\
    &\geq& C(\Lambda,d,s)r^{-d-2s} \int_{I_{r}^{\ominus}} \int_{B_r} \int_{B_r} w_{+}(t,x) w_{-}(t,y) \operatorname{d}\! x \operatorname{d}\! y \operatorname{d}\! t\\
     &&+\Lambda^{-1}\int_{I_{r}^{\ominus}} \int_{B_{r}} \int_{B_{r}} \frac{\left|w_+(t,x)-w_+(t,y)\right|^2}{|x - y|^{d+2s}} \operatorname{d}\!y \operatorname{d}\!x\operatorname{d}\! t\\
    && -C(\Lambda,d,s)(\frac{r+\rho_2}{\rho_2})^2\rho_2^{-2s}\int_{I_{r+\rho_1}^{\ominus}} \int_{B_{r+\rho_2}} w^2_+(t,x)\operatorname{d}\!x \operatorname{d}\! t\,.
\end{eqnarray*}
Regarding the estimate of $IV_2$, we once again utilize the ellipticity condition \eqref{ellip} to derive
\begin{eqnarray*}
  -IV_2 &\leq& 2\Lambda\int_{I_{r+\rho_1}^{\ominus}} \int_{\mathbb{R}^d\setminus B_{r+\rho_2}} \int_{B_{r+\rho_2}} (u(t,x)-u(t,y))_+w_+(t,y)\phi^2(y)\eta^2(t) |x - y|^{-d-2s}  \operatorname{d}\!y \operatorname{d}\!x \operatorname{d}\! t\\
  &\leq& 2\Lambda\int_{I_{r+\rho_1}^{\ominus}} \int_{B_{r+\frac{\rho_2}{2}}}w_+(t,y) \left(\int_{\mathbb{R}^d\setminus B_{r+\rho_2}} w_+(t,x)|y - x|^{-d-2s}  \operatorname{d}\!x\right) \operatorname{d}\!y \operatorname{d}\! t\\
&\leq& C(\Lambda,d,s)(\frac{r+\rho_2}{\rho_2})^d\rho_2^{-2s}\int_{I_{r+\rho_1}^{\ominus}} \left(\int_{B_{r+\rho_2}}w_+(t,x) \operatorname{d}\!x\right) \operatorname{Tail}(w_+(t);x_0,r+\rho_2) \operatorname{d}\! t\,.\\
\end{eqnarray*}
Finally, inserting the estimates for $I$ to $IV$ into \eqref{weaksol2}, we conclude that the Caccioppoli inequality \eqref{eq:Cacc} is satisfied. Hence, we complete the proof of Lemma \ref{lemma:Cacc}\,.
\end{proof}

\section{Local boundedness}\label{section3}

Equipped with the Caccioppoli inequality with tails, in this section, we demonstrate the local boundedness of weak solutions to the space-time nonlocal parabolic equation \eqref{eq1} by proving Theorem \ref{lemma:locbd}.

\begin{proof}[Proof of Theorem \ref{lemma:locbd}]
We first argue that
\begin{eqnarray}\label{bdd}
  \sup_{Q_{\frac{R}{2}}} |u| &\leq&
  C \delta^{-\frac{d+2}{2s}}\left(\fint_{Q_R} |u(t,x)|^2 \operatorname{d}\!x\operatorname{d}\!t\right)^{\frac{1}{2}} \nonumber\\
  &&+ \delta\left(\sup_{t\in I_R^{\ominus}} \operatorname{Tail}(u(t);x_0,\frac{R}{2})+\sup_{x\in B_R} \operatorname{tail}(u(x);t_0,\frac{R}{2})+R^s\|f\|_{L^\infty(I_R^{\ominus};L^{\frac{d}{s}}(B_R))}\right)
\end{eqnarray}
for any $\delta\in(0,1]$, where $C=C(d,s,\Lambda)$.

Let $r,\,\rho>0$ such that $\frac{R}{2}\leq r\leq R$  and $ \rho \leq r\leq r+\rho\leq R$.
Selecting the cut-off functions in the proof of Caccioppoli inequality as $\phi \in C_0^{\infty}(B_{r+\frac{\rho}{2}})$ with $\phi \equiv 1$ in $B_{r}$, $0 \leq \phi\leq 1$, and $|\nabla \phi| \leq C \rho^{-1}$, and $\eta \in C^{\infty}((-\infty,t_0);[0,1])$ with $\eta \equiv 1$ in $I_r^\ominus$, $\eta \equiv 0$ in $(-\infty,t_0-(r+\frac{\rho}{2})^2]$, and $|\eta'| \leq C \rho^{-2}$.
By analogy to the proof of Lemma \ref{lemma:Cacc}\,, we can deduce that
\begin{eqnarray}\label{bdd1}
 &&  \int_{I_{r+\rho}^{\ominus}}\int_{\R^d} \int_{\R^d}\frac{\left(w_{+}(t,x)\phi(x)-w_{+}(t,y)\phi(y)\right))^2}{|x-y|^{d+2s}}\eta^2(t)\operatorname{d}\! x \operatorname{d}\! y\operatorname{d}\!t \nonumber\\
 &&+\int_{B_{r}} \int_{\R} \int_{\R}\frac{\left(w_{+}(t,x)\eta(t)-w_{+}(\tau,x)\eta(\tau)\right)^2}{|t-\tau|^{1+s}}\operatorname{d}\! \tau \operatorname{d}\! t\operatorname{d}\!x\nonumber\\
 &\leq&  C \left(\frac{r+\rho}{\rho}\right)^d\rho^{-2s}\int_{I_{r + \rho}^{\ominus}} \int_{B_{r+\rho}} w^2_{+}(t,x) \operatorname{d}\! x \operatorname{d}\! t
 + C\left(\frac{r+\rho}{\rho}\right)^d \rho^{-2s} \int_{B_{r+\rho}} \int_{I_{r+\rho}^{\ominus}} w_{+}(t,x) \operatorname{d}\!t\operatorname{d}\!x\nonumber\\
 &&\times \left(\sup_{t\in I_{r+\rho}^{\ominus}}\operatorname{Tail}(w_{+}(t);x_0,r+\rho)+
 \sup_{x\in B_{r+\rho}}\operatorname{tail}(w_{+}(x);t_0,r+\rho)\right)\nonumber\\
 &&+C\int_{I_{r+\rho}^{\ominus}} \int_{B_{r+\rho}}f(t,x)w_+(t,x)\phi^2(x)\eta^2(t)\operatorname{d}\!x \operatorname{d}\! t.
\end{eqnarray}
Applying H\"{o}lder’s inequality, Young’s inequality, and fractional Sobolev inequality given by Lemma \ref{localfracsobolev} to estimate the last term in \eqref{bdd1}, we
calculate that
\begin{eqnarray*}
  &&\int_{I_{r+\rho}^{\ominus}} \int_{B_{r+\rho}}f(t,x)w_+(t,x)\phi^2(x)\eta^2(t)\operatorname{d}\!x \operatorname{d}\! t\\
  &\leq& \int_{I_{r+\rho}^{\ominus}} \left(\int_{B_{r+\rho}}\left(w_+(t,x)\phi(x)\right)^{\frac{2d}{d-2s}}\operatorname{d}\! x\right)^{\frac{d-2s}{2d}}\left(\int_{B_{r+\rho}}|f(t,x)|^{\frac{2d}{d+2s}}\chi_{\{u>l\}}\operatorname{d}\! x\right)^{\frac{d+2s}{2d}}\eta^2(t)\operatorname{d}\! t  \\
   &\leq&  \varepsilon\int_{I_{r+\rho_1}^{\ominus}} \left(\int_{B_{r+\rho}}\left(w_+(t,x)\phi(x)\right)^{\frac{2d}{d-2s}}\operatorname{d}\! x\right)^{\frac{d-2s}{d}}\eta^2(t)\operatorname{d}\! t\\
   &&+C(\varepsilon)\int_{I_{r+\rho}^{\ominus}}\left(\int_{B_{r+\rho_2}}|f(t,x)|^{\frac{2d}{d+2s}}\chi_{\{u>l\}}\operatorname{d}\! x\right)^{\frac{d+2s}{d}}\operatorname{d}\! t\\
   &\leq&C_0(d,s)\varepsilon\left(\int_{I_{r+\rho}^{\ominus}} \int_{B_{r+\rho}} \int_{B_{r+\rho}} \frac{\left(w_+(t,x)\phi(x)-w_+(t,y)\phi(y)\right)^2}{|x - y|^{d+2s} } \operatorname{d}\!y \operatorname{d}\!x\eta^2(t) \operatorname{d}\! t\right.\\
   &&\left.+(r+\rho)^{-2s}\int_{I_{r+\rho}^{\ominus}} \int_{B_{r+\rho}} w_+^2(t,x)\operatorname{d}\! x\operatorname{d}\! t\right)\\
   &&+C(\varepsilon)\int_{I_{r+\rho}^{\ominus}}\left(\int_{B_{r+\rho}}|f(t,x)|^{\frac{d}{s}}
   \operatorname{d}\! x\right)^{\frac{2s}{d}}\left( \int_{B_{r+\rho}}
   \chi_{\{u>l\}}\operatorname{d}\! x\right)\operatorname{d}\! t\\
   &\leq&C_0\varepsilon\left(\int_{I_{r+\rho}^{\ominus}} \int_{B_{r+\rho}} \int_{B_{r+\rho}} \frac{\left(w_+(t,x)\phi(x)-w_+(t,y)\phi(y)\right)^2}{|x - y|^{d+2s} } \operatorname{d}\!y \operatorname{d}\!x\eta^2(t) \operatorname{d}\! t\right.\\
   &&\left.+(\frac{r+\rho}{\rho})^d\rho^{-2s}\int_{I_{r+\rho}^{\ominus}} \int_{B_{r+\rho}} w_+^2(t,x)\operatorname{d}\! x\operatorname{d}\! t\right)\\
   &&+C(\varepsilon)\|f\|^2_{L^{\infty}(I_{r+\rho}^{\ominus};L^{\frac{d}{s}}(B_{r+\rho}))}
   \int_{I_{r+\rho}^{\ominus}} \int_{B_{r+\rho}}
   \chi_{\{u>l\}}\operatorname{d}\! x\operatorname{d}\! t\,.
\end{eqnarray*}
By selecting $\varepsilon=\frac{1}{2C_0C}$, and substituting the aforementioned inequality into \eqref{bdd1}, we derive
\begin{eqnarray*}
 &&  \int_{I_{r}^{\ominus}}\int_{\R^d} \int_{\R^d}\frac{\left(w_{+}(t,x)\phi(x)-w_{+}(t,y)\phi(y)\right))^2}{|x-y|^{d+2s}}\operatorname{d}\! x \operatorname{d}\! y\operatorname{d}\!t \nonumber\\
 &&+\int_{B_{r}} \int_{\R} \int_{\R}\frac{\left(w_{+}(t,x)\eta(t)-w_{+}(\tau,x)\eta(\tau)\right)^2}{|t-\tau|^{1+s}}\operatorname{d}\! \tau \operatorname{d}\! t\operatorname{d}\!x\nonumber\\
 &\leq&  C \left(\frac{r+\rho}{\rho}\right)^d\rho^{-2s}\int_{I_{r + \rho}^{\ominus}} \int_{B_{r+\rho}} w^2_{+}(t,x) \operatorname{d}\! x \operatorname{d}\! t
 + C\left(\frac{r+\rho}{\rho}\right)^d \rho^{-2s} \int_{B_{r+\rho}} \int_{I_{r+\rho}^{\ominus}} w_{+}(t,x) \operatorname{d}\!t\operatorname{d}\!x\nonumber\\
 &&\times \left(\sup_{t\in I_{r+\rho}^{\ominus}}\operatorname{Tail}(w_{+}(t);x_0,r+\rho)+
 \sup_{x\in B_{r+\rho}}\operatorname{tail}(w_{+}(x);t_0,r+\rho)\right)\nonumber\\
 &&+C\|f\|^2_{L^{\infty}(I_{r+\rho}^{\ominus};L^{\frac{d}{s}}(B_{r+\rho}))}
   \int_{I_{r+\rho}^{\ominus}} \int_{B_{r+\rho}}
   \chi_{\{u>l\}}\operatorname{d}\! x\operatorname{d}\! t\,.
\end{eqnarray*}
By further applying the fractional Sobolev embedding inequality as established in Lemma \ref{fracsobolev}, respectively to time and space, we obtain
\begin{eqnarray*}
 &&  \int_{I_{r}^{\ominus}}\left(\int_{B_r} w^{\frac{2d}{d-2s}}_{+}(t,x)\operatorname{d}\! x \right)^{\frac{d-2s}{d}}\operatorname{d}\!t +\int_{B_r}\left(\int_{I_{r}^{\ominus}} w^{\frac{2}{1-s}}_{+}(t,x) \operatorname{d}\! t\right)^{1-s}\operatorname{d}\!x\nonumber\\
 &\leq&  C \left(\frac{r+\rho}{\rho}\right)^d\rho^{-2s}\int_{I_{r + \rho}^{\ominus}} \int_{B_{r+\rho}} w^2_{+}(t,x) \operatorname{d}\! x \operatorname{d}\! t
 + C\left(\frac{r+\rho}{\rho}\right)^d \rho^{-2s} \int_{B_{r+\rho}} \int_{I_{r+\rho}^{\ominus}} w_{+}(t,x) \operatorname{d}\!t\operatorname{d}\!x\nonumber\\
 &&\times \left(\sup_{t\in I_{r+\rho}^{\ominus}}\operatorname{Tail}(w_{+}(t);x_0,r+\rho)+
 \sup_{x\in B_{r+\rho}}\operatorname{tail}(w_{+}(x);t_0,r+\rho)\right)\\
  &&+C\|f\|^2_{L^{\infty}(I_{r+\rho}^{\ominus};L^{\frac{d}{s}}(B_{r+\rho}))}
   \int_{I_{r+\rho}^{\ominus}} \int_{B_{r+\rho}}
   \chi_{\{u>l\}}\operatorname{d}\! x\operatorname{d}\! t\,.
\end{eqnarray*}
Next, applying H\"{o}lder's inequality twice and combining with Minkowski's inequality, we directly calculate that
\begin{eqnarray}\label{bdd2}
   \|w^2_+\|_{L^p(I_{r}^{\ominus}\times B_{r})}
 &\leq&  C \sigma(r,\rho)\|w^2_+\|_{L^1(I_{r+\rho}^{\ominus}\times B_{r+\rho})}
 + C\sigma(r,\rho)\|w_+\|_{L^1(I_{r+\rho}^{\ominus}\times B_{r+\rho})}\nonumber\\
 &&\times \left(\sup_{t\in I_{r+\rho}^{\ominus}}\operatorname{Tail}(w_{+}(t);x_0,r+\rho)+
 \sup_{x\in B_{r+\rho}}\operatorname{tail}(w_{+}(x);t_0,r+\rho)\right)\nonumber\\
  &&+C\|f\|^2_{L^{\infty}(I_{r+\rho}^{\ominus};L^{\frac{d}{s}}(B_{r+\rho}))}
   \int_{I_{r+\rho}^{\ominus}} \int_{B_{r+\rho}}
   \chi_{\{u>l\}}\operatorname{d}\! x\operatorname{d}\! t\,,
\end{eqnarray}
where $p:=\frac{d+2}{d+2(1-s)}>1$ and $\sigma(r,\rho):=\left(\frac{r+\rho}{\rho}\right)^d\rho^{-2s}$. For $l>0$, we rename $w_l(t,x):=(u(t,x)-l)_+$ and set
$$|A(l,r)|:= \int_{I_{r}^{\ominus}} \left|\{x\in B_r \mid u(t,x)>l\}\right|\operatorname{d}\! t\,,$$
then a combination of \eqref{bdd2} with H\"{o}lder's inequality yields that
\begin{eqnarray}\label{bdd3}
\|w^2_l\|_{L^1(I_{r}^{\ominus}\times B_{r})}
&\leq& |A(l,r)|^{\frac{1}{p'}}
\|w^2_l\|_{L^p(I_{r}^{\ominus}\times B_{r})}\nonumber\\
&\leq&C  |A(l,r)|^{\frac{1}{p'}}\sigma(r,\rho)\left[\|w^2_l\|_{L^1(I_{r+\rho}^{\ominus}\times B_{r+\rho})}\right.\nonumber\\
&& + \|w_l\|_{L^1(I_{r+\rho}^{\ominus}\times B_{r+\rho})} \left(\sup_{t\in I_{r+\rho}^{\ominus}}\operatorname{Tail}(u(t);x_0,r+\rho)+
 \sup_{x\in B_{r+\rho}}\operatorname{tail}(u(x);t_0,r+\rho)\right)\nonumber\\
 &&\left.\rho^{2s}|A(l,r+\rho)|\|f\|^2_{L^{\infty}(I_{r+\rho}^{\ominus};L^{\frac{d}{s}}(B_{r+\rho}))}\right]\,.
\end{eqnarray}
Let $k\in(0,l)$ be arbitrary, then there holds that
\begin{eqnarray*}
  \|w^2_l\|_{L^1(I_{r+\rho}^{\ominus}\times B_{r+\rho})} &\leq& \|w^2_k\|_{L^1(I_{r+\rho}^{\ominus}\times B_{r+\rho})}\,, \\
  \|w_l\|_{L^1(I_{r+\rho}^{\ominus}\times B_{r+\rho})} &\leq& \frac{\|w^2_k\|_{L^1(I_{r+\rho}^{\ominus}\times B_{r+\rho})}}{l-k}\,, \\
   |A(l,r)|\leq|A(l,r+\rho)| &\leq&\frac{\|w^2_k\|_{L^1(I_{r+\rho}^{\ominus}\times B_{r+\rho})}}{(l-k)^2}\,,
\end{eqnarray*}
where the Chebyshev's inequality is employed in the last estimate.
Substituting such three inequalities into \eqref{bdd3}, we derive
\begin{eqnarray}\label{bdd4}
\|w^2_l\|_{L^1(I_{r}^{\ominus}\times B_{r})}
&\leq&C(l-k)^{-\frac{2}{p'}}\sigma(r,\rho) \|w^2_k\|_{L^1(I_{r+\rho}^{\ominus}\times B_{r+\rho})}^{1+\frac{1}{p'}}\nonumber\\
&&\times \left(1+\frac{\displaystyle\sup_{t\in I_{r+\rho}^{\ominus}}\operatorname{Tail}(u(t);x_0,r+\rho)+
 \sup_{x\in B_{r+\rho}}\operatorname{tail}(u(x);t_0,r+\rho)}{l-k}\right.\nonumber\\
 &&\left.+\frac{\rho^{2s}\|f\|^2_{L^{\infty}(I_{r+\rho}^{\ominus};L^{\frac{d}{s}}(B_{r+\rho}))}}{(l-k)^2}\right),
\end{eqnarray}
where the positive constant $C$ only depends on $d$, $s$, and $\Lambda$.

To proceed, we establish the De Giorgi iteration scheme to iterate the key estimate \eqref{bdd4}. For any $i\in \mathbb{N}$, we define
\begin{eqnarray*}
   && l_i=M(1-2^{-i}),\, \, w_i=(u-l_i)_+, \,\, A_i=\|w^2_i\|_{L^1(I_{r_i}^{\ominus}\times B_{r_i})},\\
   && \rho _i=2^{-i-1}R, \,\, r_0=R, \,\, r_{i+1}=r_i-\rho_{i+1}=\frac{R}{2}\left(1+\big(\frac{1}{2}\big)^{i+1}\right),
\end{eqnarray*}
where $M>0$ is to be determined later. It is easy to verify that $\rho_i\leq\frac{R}{2}\leq r_i\leq r_i+\rho_i\leq R$, and $\rho_i\searrow 0$, $r_i\searrow \frac{R}{2}$, $l_i\nearrow M$ as $i\rightarrow\infty$.
Hence, it follows from \eqref{bdd4} that
\begin{eqnarray*}
  A_{i+1} &\leq&\frac{C\sigma(r_{i+1},\rho_{i+1})}{(l_{i+1}-l_{i})^{\frac{2}{p'}}} A_{i}^{1+\frac{1}{p'}}\\
  &&\times\left(1+
  \frac{\displaystyle\sup_{t\in I_{r_{i+1}+\rho_{i+1}}^{\ominus}}\operatorname{Tail}(u(t);x_0,r_{i+1}+\rho_{i+1})+
 \sup_{x\in B_{r_{i+1}+\rho_{i+1}}}\operatorname{tail}(u(x);t_0,r_{i+1}+\rho_{i+1})}{l_{i+1}-l_{i}}\right.\\
 &&\left.+\frac{\rho_{i+1}^{2s}\|f\|^2_{L^{\infty}(I_{r_{i+1}+\rho_{i+1}}^{\ominus};L^{\frac{d}{s}}(B_{r_{i+1}+\rho_{i+1}}))}}{(l_{i+1}-l_i)^2}
 \right) \\
 &\leq&\frac{C2^{(2s+d+\frac{2}{p'}+2)i}}{R^{2s}M^{\frac{2}{p'}}} A_{i}^{1+\frac{1}{p'}}\\
 &&\times\left(1+
  \frac{\displaystyle\sup_{t\in I_{R}^{\ominus}}\operatorname{Tail}(u(t);x_0,\frac{R}{2})+
 \sup_{x\in B_{R}}\operatorname{tail}(u(x);t_0,\frac{R}{2})}{M}+\frac{R^{2s}\|f\|^2_{L^\infty(I_R^\ominus;L^{\frac{d}{s}}(B_R))}}{M^2}\right) \,.
\end{eqnarray*}
For any fixed $\delta\in(0,1]$, we select
\begin{equation*}
  M\geq \delta^2\left(\displaystyle\sup_{t\in I_{R}^{\ominus}}\operatorname{Tail}(u(t);x_0,\frac{R}{2})+
 \sup_{x\in B_{R}}\operatorname{tail}(u(x);t_0,\frac{R}{2})\right)+\delta R^{s}\|f\|_{L^\infty(I_R^\ominus;L^{\frac{d}{s}}(B_R))}\,.
\end{equation*}
Denoting $b=2^{2s+d+\frac{2}{p'}+2}>1$, then we deduce that
\begin{equation}\label{bdd5}
  A_{i+1} \leq\frac{C_1 }{\delta^2 R^{2s}M^{\frac{2}{p'}}} b^i A_{i}^{1+\frac{1}{p'}},
\end{equation}
where $C_1$ is a positive constant depending exclusively on $d$, $s$, and $\Lambda$. We further choose
\begin{equation*}
  M=\delta^2\left(\displaystyle\sup_{t\in I_{R}^{\ominus}}\operatorname{Tail}(u(t);x_0,\frac{R}{2})+
 \sup_{x\in B_{R}}\operatorname{tail}(u(x);t_0,\frac{R}{2})\right)+\delta R^{s}\|f\|_{L^\infty(I_R^\ominus;L^{\frac{d}{s}}(B_R))}+C_1^{\frac{p'}{2}}b^{\frac{p'^2}{2}}\delta^{-p'}R^{-sp'}A_0^{\frac{1}{2}},
\end{equation*}
then it leads to
\begin{equation}\label{bdd6}
  A_{0} \leq\left(\frac{C_1}{\delta^2 R^{2s}M^{\frac{2}{p'}}}\right)^{-p'}
   b^{-p'^2}\,.
\end{equation}
In terms of \eqref{bdd5} and \eqref{bdd6}, we utilize the well-known iteration lemma given by Lemma \ref{iterate1} to derive $\displaystyle\lim_{i\rightarrow\infty}A_i=0$, which implies that
\begin{eqnarray*}
  \sup_{I_{\frac{R}{2}}^{\ominus}\times B_{\frac{R}{2}}}u(t,x) &\leq& M \\
   &\leq&C\delta^{-p'}\left(R^{-2sp'}\int_{I_{R}^{\ominus}} \int_{B_{R}} u^2(t,x) \operatorname{d}\! x \operatorname{d}\! t\right)^{\frac{1}{2}}\\
   &&+
   \delta\left(\displaystyle\sup_{t\in I_{R}^{\ominus}}\operatorname{Tail}(u(t);x_0,\frac{R}{2})+
 \sup_{x\in B_{R}}\operatorname{tail}(u(x);t_0,\frac{R}{2})+R^{s}\|f\|_{L^\infty(I_R^\ominus;L^{\frac{d}{s}}(B_R))}\right)\,.
\end{eqnarray*}
Since the conjugate of $p$, denoted as $p'=\frac{d+2}{2s}$, then it follows that
\begin{eqnarray}\label{bdd7}
  \sup_{ Q_{\frac{R}{2}}}u(t,x)
   &\leq&C\delta^{-\frac{d+2}{2s}}\left(\fint_{Q_{R}} u^2(t,x) \operatorname{d}\! x \operatorname{d}\! t\right)^{\frac{1}{2}}\nonumber\\
   &&+
   \delta\left(\displaystyle\sup_{t\in I_{R}^{\ominus}}\operatorname{Tail}(u(t);x_0,\frac{R}{2})+
 \sup_{x\in B_{R}}\operatorname{tail}(u(x);t_0,\frac{R}{2})+ R^{s}\|f\|_{L^\infty(I_R^\ominus;L^{\frac{d}{s}}(B_R))}\right)\,.\nonumber\\
\end{eqnarray}
Moreover, it is feasible to show that the weak solution $u$ is locally bounded from below, fulfilling an estimate analogous to \eqref{bdd7}. The proof follows the same arguments as before, instead of dealing with the previously defined sequence $w_i$, we now consider $w_i=(u-l_i)_-$, and use the Caccippoli inequality \eqref{eq:Cacc} for $w_-$. Hence, we confirm the validity of the assertion \eqref{bdd}.

Based on \eqref{bdd}, we further calculate that
\begin{eqnarray}\label{bdd8}
  \sup_{Q_{\frac{R}{2}}} |u| &\leq& C\delta \sup_{Q_{R}} |u|+
  C\delta^{-\frac{d+2}{2s}}\left(\fint_{Q_{R}} u^2(t,x) \operatorname{d}\! x \operatorname{d}\! t\right)^{\frac{1}{2}}\nonumber\\
   &&+
   \delta\left(\displaystyle\sup_{t\in I_{R}^{\ominus}}\operatorname{Tail}(u(t);x_0,R)+
 \sup_{x\in B_{R}}\operatorname{tail}(u(x);t_0,R)+ R^{s}\|f\|_{L^\infty(I_R^\ominus;L^{\frac{d}{s}}(B_R))}\right)\,,
\end{eqnarray}
where $C=C(d,s,\Lambda)>0$. Let $\frac{R}{2}\leq\varrho_1<\varrho_2\leq R$, performing a standard covering argument to \eqref{bdd8} and combining with Young's inequality, we derive
\begin{eqnarray*}
  \sup_{Q_{\varrho_1}} |u| &\leq& C_2\delta \sup_{Q_{\varrho_2}} |u|+
  C(\varrho_2-\varrho_1)^{-(d+2)}\int_{Q_{R}} |u(t,x)| \operatorname{d}\! x \operatorname{d}\! t\nonumber\\
   &&+
   C\left(\displaystyle\sup_{t\in I_{R}^{\ominus}}\operatorname{Tail}(u(t);x_0,R)+
 \sup_{x\in B_{R}}\operatorname{tail}(u(x);t_0,R)+ R^{s}\|f\|_{L^\infty(I_R^\ominus;L^{\frac{d}{s}}(B_R))}\right)\,,
\end{eqnarray*}
where the positive constants $C_2=C_2(d,s,\Lambda)$ and $C=C(d,s,\Lambda,\delta)$.
If we choose $\delta=\frac{1}{2C_2}$, then
\begin{eqnarray*}
  \sup_{Q_{\varrho_1}} |u| &\leq& \frac{1}{2} \sup_{Q_{\varrho_2}} |u|+
  C(\varrho_2-\varrho_1)^{-(d+2)}\int_{Q_{R}} |u(t,x)| \operatorname{d}\! x \operatorname{d}\! t\nonumber\\
   &&+
   C\left(\displaystyle\sup_{t\in I_{R}^{\ominus}}\operatorname{Tail}(u(t);x_0,R)+
 \sup_{x\in B_{R}}\operatorname{tail}(u(x);t_0,R)+ R^{s}\|f\|_{L^\infty(I_R^\ominus;L^{\frac{d}{s}}(B_R))}\right)\,,
\end{eqnarray*}
where $C=C(d,s,\Lambda)>0$. Finally, we absorb the first term on the right side by using an iteration Lemma \ref{iterate3} to deduce that the local boundedness estimate \eqref{locbd} holds. Hence, we complete the proof of Theorem \ref{lemma:locbd}.
\end{proof}

\section{Local H\"{o}lder regularity}\label{section4}

This section is devoted to the proof of the local H\"{o}lder continuity of weak solutions for the space-time nonlocal parabolic equations \eqref{eq1}, namely Theorem \ref{prop:Holderq}. Having  Caccioppoli inequality and local boundedness at our disposal,
it is crucial for us to further establish a growth lemma. For ease of notation, we omit the vertex $(x_0,t_0)$ from the parabolic cylinder $Q_{R}(x_0,t_0)$ in the following proof.
Now we first introduce a De Giorgi type of lemma.

\begin{lemma}
\label{lemma:growth-lemma1}
Let $d>2$, $q=\frac{d}{s}$ and $u$ be a local weak supersolution to \eqref{eq1} in $I\times\Omega$.
For every $R>0$ and $(t_0,x_0)\in I\times\Omega$ with $I^\ominus_{2R}(t_0)\times B_{2R}(x_0)\Subset I\times\Omega$, any $\delta \in (0,1]$ and $H > 0$, there exists $\nu \in (0,1)$ depending only on $d,\,s,\,\Lambda,\,\delta$, such that if  $u \geq 0$ in $Q_{2R}(t_0,x_0)$, as well as
\begin{equation}\label{ularge}
\left| \left\{ u \leq H \right\} \cap I^{\ominus}_{\delta R}(t_0) \times B_R (x_0) \right| \leq \nu |I^{\ominus}_{\delta R} (t_0)\times B_R(x_0)|,
\end{equation}
\begin{equation}\label{eq:Tail-small}
\left(\fint_{I_{2R}^{\ominus}(t_0)} \operatorname{Tail}(u_-(t);x_0,2R)^q \operatorname{d}\! t \right)^{\frac{1}{q}} \leq H,
\end{equation}
\begin{equation}\label{eq:tail-small}
\left(\fint_{B_{2R}(x_0)} \operatorname{tail}(u_-(x);t_0,2\delta R)^q \operatorname{d}\! x \right)^{\frac{1}{q}}\leq H,
\end{equation}
and
\begin{equation}\label{eq:f-small}
R^s\left(\fint_{I_{2R}^{\ominus}(t_0)}\int_{B_{2R}(x_0)} |f(t,x)|^q \operatorname{d}\! x \operatorname{d}\! t \right)^{\frac{1}{q}}\leq H,
\end{equation}
then it holds that
\begin{equation*}
u(t,x) \geq \frac{H}{2} ~~ \text{ in } I^{\ominus}_{\frac{\delta R}{2}}(t_0) \times B_{\frac{R}{2}}(x_0).
\end{equation*}
\end{lemma}

\begin{proof}
For $i\in\mathbb{N_+}$, we define the following sequences
\begin{eqnarray*}
&& k_i = \frac{H}{2} + \frac{H}{2^i}, \,\, \hat{k}_i = \frac{k_i + k_{i+1}}{2}, \,\, w_i = (u-k_{i})_-, \\
&& R_i = \frac{R}{2} + \frac{R}{2^i}, \,\, \hat{R}_i = \frac{R_i + R_{i+1}}{2}\\
&& A_i= \frac{\left| \left\{ u \leq k_i \right\} \cap I^{\ominus}_{\delta R_i} \times B_{R_i} \right|}{|I^{\ominus}_{\delta R_i} \times B_{R_i}|}, \,\, B_i = \left( \fint_{I_{\delta R_i}^{\ominus}} \left( \frac{|\{ u(t) \leq k_i \} \cap B_{R_i}|}{|B_{R_i}|} \right)\operatorname{d}\! t \right)^{\frac{q-1}{q(1+\kappa)}},
\end{eqnarray*}
where $\kappa=\frac{2qs-d-2}{q(d+2-2s)}>0$ by virtue of $d>2$.
Combining the Caccioppoli inequality \eqref{eq:Cacc} with H\"{o}lder's inequality, the conditions \eqref{eq:Tail-small}-\eqref{eq:f-small} and $q=\frac{d}{s}>1$, we directly compute
\begin{eqnarray}\label{growth1-1}
 &&  \int_{I_{\delta R_{i+1}}^{\ominus}}\int_{B_{R_{i+1}}} \int_{B_{R_{i+1}}}\frac{|w_{i}(t,x)-w_{i}(t,y)|^2}{|x-y|^{d+2s}}\operatorname{d}\! x \operatorname{d}\! y\operatorname{d}\!t \nonumber\\
 &&+\int_{B_{R_{i+1}}}\int_{I_{\delta R_{i+1}}^{\ominus}} \int_{I_{\delta R_{i+1}}^{\ominus}} \frac{|w_{i}(t,x)-w_{i}(\tau,x)|^2}{|t-\tau|^{1+s}}\operatorname{d}\! \tau \operatorname{d}\! t\operatorname{d}\!x \nonumber \\
 &\leq&  C 2^{i(2+2s)}(\delta R)^{-2s} \int_{I_{\delta \hat{R}_i}^{\ominus}} \int_{B_{\hat{R}_i}} w^2_{i}(t,x) \operatorname{d}\! x \operatorname{d}\! t\nonumber\\
 &&+ C2^{i(d+2s)}(\delta R)^{-2s} \left[\int_{B_{\hat{R}_i}} \left(\int_{I_{\delta\hat{R}_i}^{\ominus}} w_{i}(t,x) \operatorname{d}\!t \right) \operatorname{tail}(w_{i}(x);t_0,\delta\hat{R}_i)\operatorname{d}\!x \right. \nonumber\\
 &&\left.+ \delta^{2s} \int_{I_{\delta\hat{R}_i}^{\ominus}} \left(\int_{B_{\hat{R}_i}} w_{i}(t,x) \operatorname{d}\!x \right) \operatorname{Tail}(w_{i}(t);x_0,\hat{R}_i)\operatorname{d}\!t\right]+\int_{I_{\delta\hat{R}_i}^{\ominus}}
 \int_{B_{\hat{R}_i}} w_{i}(t,x) |f(t,x)|\operatorname{d}\!x\operatorname{d}\!t\nonumber\\
 &\leq&C2^{i(d+2s)}R^{-2s}H^2|Q_{R_i}|A_i\nonumber\\
 &&+C2^{i(d+2s)}(\delta R)^{-2s}H\left[\delta^{2s} \int_{I_{\delta\hat{R}_i}^{\ominus}} \left(\int_{B_{\hat{R}_i}} \chi_{\{u\leq k_i\}}(t,x) \operatorname{d}\!x\right)\operatorname{Tail}(w_{i}(t);x_0,2R)\operatorname{d}\!t
 \right.\nonumber\\
 &&+\left.\int_{B_{\hat{R}_i}} \left(\int_{I_{\delta\hat{R}_i}^{\ominus}} \chi_{\{u\leq k_i\}}(t,x) \operatorname{d}\!t \right) \operatorname{tail}(w_{i}(x);t_0,2\delta R)\operatorname{d}\!x \right]\nonumber\\
&& +H\left(\int_{I_{\delta\hat{R}_i}^{\ominus}}
 \int_{B_{\hat{R}_i}} |f(t,x)|^q\operatorname{d}\!x\operatorname{d}\!t\right)^{\frac{1}{q}}\left(
 \int_{I_{\delta\hat{R}_i}^{\ominus}}
 \int_{B_{\hat{R}_i}}  \chi_{\{u\leq k_i\}}(t,x)\operatorname{d}\!x\operatorname{d}\!t\right)^{\frac{1}{q'}}
\nonumber\\
  &\leq&C2^{i(d+2s)}R^{-2s}H^2|Q_{R_i}|A_i+C2^{i(d+2s)}(\delta R)^{-2s}H\left[ \delta^2H|Q_{R_i}|A_i\right.\nonumber\\
&& +\delta^{2s} \int_{I_{\delta\hat{R}_i}^{\ominus}} \left(\int_{B_{\hat{R}_i}} \chi_{\{u\leq k_i\}}(t,x) \operatorname{d}\!x\right)\operatorname{Tail}(u_{-}(t);x_0,2R)\operatorname{d}\!t
\nonumber\\
 &&+\left.\int_{B_{\hat{R}_i}} \left(\int_{I_{\delta\hat{R}_i}^{\ominus}} \chi_{\{u\leq k_i\}}(t,x) \operatorname{d}\!t \right) \operatorname{tail}(u_{-}(x);t_0,2\delta R)\operatorname{d}\!x \right]\nonumber\\
 &&+H|Q_{R_i}|\delta^{2-\frac{2}{q}}B_i^{1+\kappa}R^{-s}\left(\fint_{I_{2R}^{\ominus}}
 \int_{B_{2R}} |f(t,x)|^q\operatorname{d}\!x\operatorname{d}\!t\right)^{\frac{1}{q}}\nonumber\\
 &\leq&C2^{i(d+2s)}R^{-2s}H^2|Q_{R_i}|\left(A_i+B_i^{1+\kappa}\right)\,.
\end{eqnarray}
Applying H\"{o}lder inequality repeatedly, and combining with Young's inequality, Minkowski's inequality and the fractional Sobolev embedding inequality established in Lemma \ref{localfracsobolev}, we deduce that
\begin{eqnarray*}
 && A_{i+1}|I^{\ominus}_{\delta R_{i+1}} \times B_{R_{i+1}}|\\
  &\leq& \int_{I_{\delta R_{i+1}}^{\ominus}} \int_{B_{R_{i+1}}}\frac{w_i^2}{(\hat{k}_i-k_{i+1})^2}\operatorname{d}\!x\operatorname{d}\!t  \\
   &\leq& C2^{2i}H^{-2} \left( A_{i}|I^{\ominus}_{\delta R_{i}} \times B_{R_{i}}|\right)^{\frac{1}{p'}} \left(\int_{I_{\delta R_{i+1}}^{\ominus}} \int_{B_{R_{i+1}}}w_i^{2p}\operatorname{d}\!x\operatorname{d}\!t\right)^{\frac{1}{p}} \\
   &\leq& C2^{2i}H^{-2} \left( A_{i}|I^{\ominus}_{\delta R_{i}} \times B_{R_{i}}|\right)^{\frac{1}{p'}} \left[\int_{I_{\delta R_{i+1}}^{\ominus}} \left(\int_{B_{R_{i+1}}}w_i^{\frac{2d}{d-2s}}\operatorname{d}\!x\right)^{\frac{d-2s}{d}}\operatorname{d}\!t + \int_{B_{R_{i+1}}} \left(\int_{I_{\delta R_{i+1}}^{\ominus}}w_i^{\frac{2}{1-s}}\operatorname{d}\!t\right)^{1-s}\operatorname{d}\!x \right] \\
    &\leq& C2^{2i}H^{-2} \left( A_{i}|I^{\ominus}_{\delta R_{i}} \times B_{R_{i}}|\right)^{\frac{1}{p'}} \left[(\delta R_{i+1})^{-2s}\int_{I_{\delta R_{i+1}}^{\ominus}} \int_{B_{R_{i+1}}}w_i^{2}\operatorname{d}\!x\operatorname{d}\!t \right. \\
    && + \int_{I_{\delta R_{i+1}}^{\ominus}}\int_{B_{R_{i+1}}} \int_{B_{R_{i+1}}}\frac{|w_{i}(t,x)-w_{i}(t,y)|^2}{|x-y|^{d+2s}}\operatorname{d}\! x \operatorname{d}\! y\operatorname{d}\!t \nonumber\\
 &&\left.+\int_{B_{R_{i+1}}}\int_{I_{\delta R_{i+1}}^{\ominus}} \int_{I_{\delta R_{i+1}}^{\ominus}} \frac{|w_{i}(t,x)-w_{i}(\tau,x)|^2}{|t-\tau|^{1+s}}\operatorname{d}\! \tau \operatorname{d}\! t\operatorname{d}\!x \right]\,,
\end{eqnarray*}
where $p=\frac{d+2}{d+2(1-s)}$ and $p'=\frac{d+2}{2s}$ are conjugate exponents.
Then it follows from \eqref{growth1-1} that
\begin{equation}\label{growth1-2}
  A_{i+1}\leq C_3\delta^{\frac{4s}{d+2}-2}2^{(d+2+2s)i}\left(A_i^{1+\frac{2s}{d+2}}+A_i^{\frac{2s}{d+2}}B_i^{1+\kappa}\right),
\end{equation}
where the positive constant $C_3$ only depends on $d$, $s$ and $\Lambda$.
Note that $\frac{q-1}{q(1+\kappa)}=\frac{1}{p}$ by the definitions of $\kappa$ and $p$, then in analogy with the estimates of $A_{i+1}$ in \eqref{growth1-2}, we can further obtain
\begin{equation}\label{growth1-3}
  B_{i+1}\leq C_4\delta^{\frac{4s}{d+2}-2}2^{(d+2+2s)i}\left(A_i+B_i^{1+\kappa}\right),
\end{equation}
where the constant $C_4=C_4(d,s,\Lambda)>0$.
Meanwhile, the assumption \eqref{ularge} implies that
\begin{equation*}
  A_1+B_1^{1+\kappa}=  A_1+A_1^{\frac{q-1}{q}}\leq 2A_1^{\frac{q-1}{q}}\leq  2 \nu^{\frac{q-1}{q}},
\end{equation*}
where $q=\frac{d}{s}$. From this, combining \eqref{growth1-2} with \eqref{growth1-3}, and applying a classical iteration Lemma \ref{iterate2} to derive that there exists $\nu >0$ sufficiently small, depending only on $d,\,s,\,\Lambda$, and $\delta$, such that
the sequences $\{A_i\}$ and $\{B_i\}$ tend to zero as $i\rightarrow \infty$.
Hence, we conclude that
\begin{equation*}
u(t,x) \geq \frac{H}{2} ~~ \text{ in } I^{\ominus}_{\frac{\delta R}{2}} \times B_{\frac{R}{2}},
\end{equation*}
which completes the proof of Lemma \ref{lemma:growth-lemma1}\,.
\end{proof}

In the sequel, we establish a measure shrinking lemma, where the third term of the Caccioppoli inequality \eqref{eq:Cacc} plays an important role in the proof.
\begin{lemma}
\label{lemma:growth-lemma2}
Let $q=\frac{d}{s}$, $\alpha \in (0,1]$ and let
$u$ be a local weak supersolution to \eqref{eq1} in $I\times\Omega$.
For every $R>0$ and $(t_0,x_0)\in I\times\Omega$ with $I^\ominus_{2R}(t_0)\times B_{2R}(x_0)\Subset I\times\Omega$, and any $\delta,\, \sigma \in (0,1]$ and $H > 0$, if $u \ge 0$ in $Q_{2R}(t_0,x_0)$
and
\begin{equation}
\label{eq:measure-ass1}
\left| \left\{ u(t,\cdot) \ge H \right\} \cap B_R(x_0) \right| \ge \alpha |B_R(x_0)| \,\,\mbox{for a.e.}\,\, t \in I_{\delta R}^{\ominus}(t_0),
\end{equation}
then there exists a positive constant $C$ depending only on $d,\,s,\, \Lambda$, such that either
\begin{equation}
\label{Tail-tail-large}
\left(\fint_{I_{2R}^{\ominus}(t_0)} \operatorname{Tail}(u_-(t);x_0,2R)^q \operatorname{d}\! t \right)^{\frac{1}{q}}>\sigma H,\,\, \left(\fint_{B_{2R}^{\ominus}(x_0)} \operatorname{tail}(u_-(x);t_0,2\delta R)^q \operatorname{d}\! x \right)^{\frac{1}{q}}> \sigma H,
\end{equation}
and
\begin{equation}\label{eq:f-large}
R^s\left(\fint_{I_{2R}^{\ominus}(t_0)}\int_{B_{2R}(x_0)} |f(t,x)|^q \operatorname{d}\! x \operatorname{d}\! t \right)^{\frac{1}{q}}>\sigma H,
\end{equation}
or
\begin{equation*}
\left| \left\{ u \le \frac{\sigma H}{4} \right\} \cap I^{\ominus}_{\delta R}(t_0) \times B_R (x_0)\right| \le  \frac{C\sigma}{\delta^{2} \alpha} |I^{\ominus}_{\delta R} (t_0)\times B_R(x_0)|.
\end{equation*}
\end{lemma}

\begin{proof}
Without loss of generality, we may assume that \eqref{Tail-tail-large} and \eqref{eq:f-large} do not hold true. Let $w(t,x) = u(t,x) - \frac{\sigma H}{2}$, then it is easy to verify that $w_-(t,x)\leq u_-(t,x)+\frac{\sigma H}{2}$.
Applying the Caccioppoli inequality \eqref{eq:Cacc} with $l=\frac{\sigma H}{2}$ and utilizing H\"{o}lder's inequality to the tail terms both in time and space, we deduce that
\begin{eqnarray*}
  &&R^{-d-2s}  \int_{I_{\delta R}^{\ominus}} \int_{B_R} \int_{B_R} w_-(t,x) w_+(t,y) \operatorname{d}\!x\operatorname{d}\!y\operatorname{d}\!t  \\
   &\leq&  C (\delta R)^{-2s}\int_{I_{2\delta R}^{\ominus}} \int_{B_{2R}} w^2_{-}(t,x) \operatorname{d}\! x \operatorname{d}\! t \nonumber\\
 &&+ C(\delta R)^{-2s} \int_{B_{2R}} \left(\int_{I_{2\delta R}^{\ominus}} w_{-}(t,x) \operatorname{d}\!t \right) \operatorname{tail}(w_{-}(x);t_0,2\delta R)\operatorname{d}\!x\nonumber\\
&&+ C R^{-2s} \int_{I_{2\delta R}^{\ominus}} \left(\int_{B_{2R}} w_{-}(t,x) \operatorname{d}\!x \right) \operatorname{Tail}(w_{-}(t);x_0,2R)\operatorname{d}\!t\\
&&+\int_{I_{2\delta R}^{\ominus}}\int_{B_{2R}}w_{-}(t,x)|f(t,x)|\operatorname{d}\!x\operatorname{d}\!t  \\
&\leq&C (\sigma H)^2 \left(\delta^{-2s}+\delta^{-2}+\delta^{-2+\frac{2}{q}}\right)  R^{-2s}|I^{\ominus}_{\delta R} \times B_R|\\
&\leq&C (\sigma H)^2 \delta^{-2}R^{-2s}|I^{\ominus}_{\delta R} \times B_R|.
\end{eqnarray*}
Finally, combining the condition \eqref{eq:measure-ass1} with the aforementioned estimates, we arrive at the conclusion that
\begin{eqnarray*}
\left| \left\{ u \le \frac{\sigma H}{4} \right\} \cap I^{\ominus}_{\delta R} \times B_R \right| &\le& c(\sigma H)^{-1} \int_{I_{\delta R}^{\ominus}} \int_{B_R} w_-(t,x) \operatorname{d}\!x\operatorname{d}\!t\\
&\le& c(\sigma\alpha)^{-1}  H^{-2} R^{-d} \int_{I_{\delta R}^{\ominus}} \int_{B_R} \int_{B_R} w_-(t,x) w_+(t,y) \operatorname{d}\!x\operatorname{d}\!y\operatorname{d}\!t\\
&\le&  \frac{C\sigma}{\delta^{2} \alpha} |I^{\ominus}_{\delta R} \times B_R|.
\end{eqnarray*}
Thus, we complete the proof of Lemma \ref{lemma:growth-lemma2}\,.
\end{proof}

An immediate consequence of Lemma \ref{lemma:growth-lemma1} and Lemma \ref{lemma:growth-lemma2} is the derivation of the following growth lemma, which is an essential
ingredient in proving H\"{o}lder continuity.

\begin{corollary}
\label{cor:growth-lemma}
Let $d>2$, $q=\frac{d}{s}$, $\alpha \in (0,1]$, and let $u$ be a local weak supersolution to \eqref{eq1} in $I\times\Omega$.
For every $R>0$ and $(t_0,x_0)\in I\times\Omega$ with $I^\ominus_{2R}(t_0)\times B_{2R}(x_0)\Subset I\times\Omega$, any $H > 0$, if
\begin{equation}\label{corcon1}
\left| \left\{ u(t,x) \ge H \right\} \cap I_R^\ominus(t_0)\times B_R(x_0) \right| \ge \alpha |I_R^\ominus(t_0)\times B_R(x_0)|,
\end{equation}
then there exist
$\delta,\,\theta \in (0,1)$ depending only on $d,\,s,\,\Lambda,\,\alpha$, and $t_1\in I_R^\ominus(t_0)$ with $I_{\delta R}^\ominus(t_1)\subset I_R^\ominus(t_0)$, such that
\begin{equation*}
u(t,x) \ge \frac{\theta H}{8} ~~ \text{ in } I_{\frac{\delta R}{2}}^{\ominus} (t_1) \times B_{\frac{R}{2}}(x_0)
\end{equation*}
holds, provided  $u \ge 0$ in $I^\ominus_{2R}(t_1)\times B_{2R}(x_0)\Subset I\times\Omega$,
\begin{equation}\label{Tail-tail-small}
\left(\fint_{I_{2R}^{\ominus}(t_1)} \operatorname{Tail}(u_-(t);x_0,2R)^q \operatorname{d}\! t \right)^{\frac{1}{q}} \leq \frac{\theta H}{4}, \,\,\left(\fint_{B_{2R}(x_0)} \operatorname{tail}(u_-(x);t_1,2\delta R)^q \operatorname{d}\! x \right)^{\frac{1}{q}}\leq \frac{\theta H}{4},
\end{equation}
and
\begin{equation}\label{f-small}
R^s\left(\fint_{I_{2R}^{\ominus}(t_1)}\int_{B_{2R}(x_0)} |f(t,x)|^q \operatorname{d}\! x \operatorname{d}\! t \right)^{\frac{1}{q}}\leq\frac{\theta H}{4}.
\end{equation}
\end{corollary}
\begin{proof}
Let
$$\Sigma:=\big\{t\in I_R^\ominus(t_0)\mid |\{u(\cdot, x)\geq H\}\cap B_R(x_0)|> \frac{\alpha}{2}|B_R(x_0)|\big\},$$
using the proof by contradiction in conjunction with \eqref{corcon1}, it is easy to deduce that $|\Sigma|\geq\alpha |I_R^\ominus(t_0)|$. Then there exist $t_1\in I_R^\ominus(t_0)$ and $\delta$ depending on $\alpha$ such that the time slice  $I_{\delta R}^\ominus(t_1)\subset \Sigma$ and
\begin{equation*}
\left| \left\{ u(t,\cdot) \ge H \right\} \cap B_R(x_0) \right| \ge \frac{\alpha}{2} |B_R(x_0)| \,\,\mbox{for a.e.}\,\, t \in I_{\delta R}^{\ominus}(t_1).
\end{equation*}
From this, Lemma \ref{lemma:growth-lemma2} implies that
\begin{equation*}
\left| \left\{ u \le \frac{\theta H}{4} \right\} \cap I^{\ominus}_{\delta R}(t_1) \times B_R(x_0) \right| \le  \frac{C_5\theta}{\delta^{2} \alpha} |I^{\ominus}_{\delta R}(t_1) \times B_R(x_0)|,
\end{equation*}
where the constant $C_5=C_5(d,s,\Lambda)>1$. We further
choose $\theta\leq C_5^{-1}\delta^2\alpha\nu$ such that
\begin{equation*}
\left| \left\{ u \le \frac{\theta H}{4} \right\} \cap I^{\ominus}_{\delta R}(t_1) \times B_R(x_0) \right| \le  \nu |I^{\ominus}_{\delta R}(t_1) \times B_R(x_0)|,
\end{equation*}
where $\nu$ appears in \eqref{ularge}, then it follows from the dependence of $\nu$ that $\theta$ depends only on $d,\,s,\,\Lambda,\,\alpha$.
Hence, by virtue of Lemma \ref{lemma:growth-lemma1}\,, we conclude that
\begin{equation*}
u(t,x) \ge \frac{\theta H}{8} ~~ \text{ in } I_{\frac{\delta R}{2}}^{\ominus}(t_1) \times B_{\frac{R}{2}}(x_0),
\end{equation*}
which completes the proof of Corollary \ref{cor:growth-lemma}\,.
\end{proof}

We are now ready to prove the H\"{o}lder regularity established in Theorem \ref{prop:Holderq}.

\begin{proof}[Proof of Theorem \ref{prop:Holderq}]		
The proof goes by constructing two sequences $\{M_j\}$ and $\{m_j\}$ that are non-increasing and non-decreasing, respectively, and to select a small $\gamma \in (0,1)$ and a large $\nu > 1$ such that
\begin{equation}
\label{eq:osc-decay}
m_j \le u \le M_j ~~ \text{ in } Q_{\nu^{-j} R}(t_0,x_0), ~~ \text{ and }~~ M_j - m_j = L \nu^{-\gamma j},
\end{equation}
for any $j \in \mathbb{N}$,
where
\begin{eqnarray*}
L :&=& C_6 \Vert u \Vert_{L^{\infty}(Q_R(t_0,x_0))} + \left( \fint_{I_{2R}^{\ominus}(t_0)} \operatorname{Tail}(u(t);x_0,2R)^{\frac{d}{s}} \operatorname{d}\! t  \right)^{\frac{s}{d}}+\left( \fint_{B_{2R}(x_0)} \operatorname{tail}(u(x);t_0,2R)^{\frac{d}{s}}  \operatorname{d}\! x \right)^{\frac{s}{d}}\\
&&R^s\left(\fint_{I_{2R}^{\ominus}(t_0)}\int_{B_{2R}(x_0)} |f(t,x)|^{\frac{d}{s}}  \operatorname{d}\! x \operatorname{d}\! t \right)^{\frac{s}{d}}
\end{eqnarray*}
for some positive constant $C_6$ to be determined later. Once \eqref{eq:osc-decay} is achieved, we immediately obtain the desired result by using the definition of $L$ and Lemma \ref{lemma:locbd}\,.

For any fixed $j_0$, if we set $M_j = \frac{\nu^{-\gamma j}L}{2}$ and $m_j = - \frac{\nu^{-\gamma j}L}{2}$, and choose $C_6 \ge 2 \nu^{\gamma j_0}$, then it is not difficult to verify that \eqref{eq:osc-decay} is valid for every $j\leq j_0$.
In order to show that \eqref{eq:osc-decay} also holds true for $j>j_0$, we need to construct appropriate $M_{j+1}$ and $m_{j+1}$, and proceed by induction. Assuming that \eqref{eq:osc-decay} is valid up to some $j\geq j_0$, we will show that \eqref{eq:osc-decay} must hold true for $j+1$. To proceed,
without loss of generality, we suppose that
\begin{equation}
\label{holder1}
\big|\{ u(t ,x) \ge m_j + \frac{M_j - m_j}{2} \} \cap Q_{\nu^{-j} R}(\bar t,x_0)\big| \ge \frac{1}{2}|Q_{\nu^{-j}R}(\bar {t},x_0)|
\end{equation}
for some $\bar{t}\in (t_0,t_0+(\nu^{-j}R)^2)$, such that the above $t_0$ acts as $t_1$ in Corollary \ref{cor:growth-lemma}\,.
Through a straightforward calculation, it becomes evident that
\begin{equation*}
\left(\fint_{I_{2\nu^{-j}R}^{\ominus}(t_0)} \operatorname{Tail}((u-m_j)_-(t);x_0,2\nu^{-j}R)^{\frac{d}{s}} \operatorname{d}\! t \right)^{\frac{s}{d}} \leq \frac{\theta (M_j - m_j)}{8},
\end{equation*}
\begin{equation*}
\left(\fint_{B_{2\nu^{-j}R}(x_0)} \operatorname{tail}((u-m_j)_-(x);t_0,2\delta\nu^{-j} R)^{\frac{d}{s}} \operatorname{d}\! x \right)^{\frac{s}{d}}\leq \frac{\theta (M_j - m_j)}{8},
\end{equation*}
and
\begin{equation*}
  (2\nu^{-j}R)^s\left(\fint_{I_{2\nu^{-j}R}^{\ominus}(t_0)}\int_{B_{2\nu^{-j}R}(x_0)} |f(t,x)|^{\frac{d}{s}}  \operatorname{d}\! x \operatorname{d}\! t \right)^{\frac{s}{d}}\leq \frac{\theta (M_j - m_j)}{8}
\end{equation*}
can fulfill by first choosing a sufficiently small $\gamma<\min\big\{\frac{4s}{d+2},\frac{(d-2)s}{d}\big\}$ that depends only on $d,\,s,\,\theta$, and then selecting a sufficiently large $\nu>\frac{2}{\delta}$ that is dependent on $d,\,s,\,\theta$.
Thus, applying Corollary \ref{cor:growth-lemma} with $u:=u-m_j$, $R:=\nu^{-j} R$, $\alpha=\frac{1}{2}$, and $H:= \frac{M_j - m_j}{2}=\frac{L\nu^{-\gamma j}}{2}$, we deduce that
\begin{equation}\label{holder2}
u(t,x) \ge m_j+\frac{\theta L\nu^{-\gamma j}}{16} ~~ \text{ in } Q_{\nu^{-(j+1)}R}(t_0,x_0),
\end{equation}
where $\theta=\theta(d,s,\Lambda)\in(0,1)$. Let
\begin{equation*}
  M_{j+1}=M_j \,\,\mbox{and}\,\, m_{j+1}=M_j-L\nu^{-\gamma(j+1)},
\end{equation*}
we further select $\gamma$ to be sufficiently small such that $\gamma\leq\log_\nu(\frac{16}{16-\theta})$ to ensure that the chosen $M_{j+1}$ and $m_{j+1}$ satisfy \eqref{eq:osc-decay}.

On the other hand, if \eqref{holder1} fails, we proceed analogously but instead apply Corollary \ref{cor:growth-lemma} with $u: = M_j - u$. In this case, setting $M_{j+1}=L\nu^{-\gamma(j+1)}+m_j$ and $m_{j+1}=m_j$ as desired, then we conclude the proof of Theorem \ref{prop:Holderq}\,.
\end{proof}

\section*{Acknowledgments} This work is supported by the National Natural Science Foundation of China (NSFC Grant No.12101452 and No.12071229).

\section*{Data availability} Data sharing is not applicable to this article as obviously no datasets were generated or analyzed during the current study.

\bibliography{bibliography}

\end{document}